\documentclass{mymathart}
\usepackage{mymathmacros}
\usepackage{ogonek}

\numberwithin{equation}{section}

\newtheoremstyle{numberedstyle}
  {9pt}
  {9pt}
  {\normalfont}
  {}
  {\bfseries}
  {.}
  {\newline}
  {}

\theoremstyle{changebreak}\usepackage[english]{babel}

\newtheorem{thm}{Theorem}[section]%
\newtheorem{lem}[thm]{Lemma}%
\newtheorem{cor}[thm]{Corollary}%
\newtheorem{observation}[thm]{Observation}%
\newtheorem{lemdef}[thm]{Lemma and Definition}

\theoremstyle{numberedstyle}
\newtheorem{defn}[thm]{Definition}%

\def\good{good}
\def\controlledset{regular set}
\usepackage{hyperref}

\newcommand{\Orbit}{\mathcal{O}}

\renewcommand{\P}{\mathcal{P}}

\newcommand{\meas}{\operatorname{meas}}

\title[Absence of line fields and Ma{\~n}\'e's theorem]{%
     Absence of line fields and Ma{\~n}\'e's theorem \\ 
      for
      non-recurrent transcendental functions}
\author{Lasse Rempe}
\address{Dept.\ of Math.\ Sciences, University of Liverpool, Liverpool L69 7ZL, UK}
\email{l.rempe@liverpool.ac.uk}
\author{Sebastian van Strien}
\address{Mathematics Institute, University of Warwick, Coventry CV4 7AL, UK}
\email{strien@maths.warwick.ac.uk}
\thanks{This research is supported by EPSRC grant EP/E017886/1.}

\newcommand{\Sing}{\operatorname{sing}}

\begin{document}

\begin{abstract}
 Let $f:\C\to\Ch$ be a transcendental meromorphic function. Suppose that
  the finite part $\P(f)\cap \C$ of the 
  postsingular set of $f$ is bounded, that $f$ has no
  recurrent critical points or wandering domains, and that
  the degree of pre-poles of $f$ is uniformly bounded.
  Then we show that $f$ supports no
  invariant line fields on its Julia set. 

 We prove this by generalizing two results about rational
  functions to the transcendental setting:
  a theorem of Ma{\~n}{\'e} \cite{mane} about 
  the branching of iterated preimages of disks, and a theorem of
  McMullen \cite[Theorem 3.17]{mcmullenrenormalization}
  regarding absence of
  invariant line fields for ``measurably transitive''
  functions. Both our theorems extend results previously obtained by
  Graczyk, Kotus and \'{S}wi\k{a}tek \cite{graczykkotusswiatek}. 
\end{abstract}

\maketitle

\section{Introduction}%
Let $f:\C\to\Ch$ be a nonconstant, nonlinear
 entire or meromorphic function. The existence of \emph{invariant line
 fields} (see Section \ref{sec:linefields})
 supported on the Julia set of $f$ is an important question which is 
 related to \emph{density of hyperbolicity}, one
 of the central problems in one-dimensional real and 
 holomorphic dynamics. It is conjectured that \emph{flexible
 Latt\`es maps} (see \cite{jacklattes}) are the only rational functions
 for which such line fields exist. 

In the setting of transcendental meromorphic functions $f:\C\to\Ch$, the
 situation is less clear. 
 Indeed, it is known \cite{alexmishaexamples} that there exist
 ``pathological'' entire transcendental 
 functions that support invariant line fields on their Julia sets.
 Also, hyperbolicity --- even structural stability --- need not be dense
 in parameter spaces of entire transcendental functions if the set 
 $S(f)\subset\Ch$ of
 \emph{singular values} of $f$ (that is, the closure of the set 
 $\Sing(f^{-1})$ of
 critical and asymptotic values) is too large. (We explicitly note that
 we include $\infty$ in $S(f)$ if it is a critical or asymptotic value.
 Often $S(f)$ denotes only the \emph{finite} singular values of $f$, but
 for the purposes of this article our convention is more convenient.) 
 It is not even
 clear how ``hyperbolicity'' should be defined when $S(f)\cap\C$ 
 is unbounded. 

 However, it is expected that such phenomena can be controlled
 when the functions in question are sufficiently tame, either
 function-theoretically or dynamically. For example,
 we believe that hyperbolicity is dense in the parameter spaces $M_f$ of
 entire functions with finitely many singular values that were defined
 in \cite{alexmisha}.

In this article, we will, on the other hand, prove the
 absence of invariant line fields for meromorphic functions under some
 strong dynamical conditions, but without making many
 function-theoretic assumptions. 
 Let us say that $f$ is \emph{nonrecurrent} if the finite part
  $\P(f)\cap \C$ of the postsingular set
    \[ \P(f) := \cl{\{f^n(s): s\in S(f), n\geq 0\}} \subset\Ch \]
  is bounded, and furthermore every critical point $c$ of $f$ is
  nonrecurrent; i.e. $c\notin \cl{\{f^n(c): n>0\}}$. 

\begin{thm}[Absence of line fields for nonrecurrent maps]
   \label{thm:nonrecurrent}
 Let $f:\C\to\Ch$ be a nonlinear and nonconstant meromorphic function. 
  Suppose that $f$
  is nonrecurrent and has no wandering domains, and that
  there is a bound on the order of all pre-poles of $f$.
  Then 
  the Julia set of $f$ supports no invariant line fields. 
\end{thm}
\begin{remark}
  If $f$ as in the statement of the
   theorem had a wandering domain, then the set of limit points of the
   iterates on this domain would have to be bounded. It is an open question
   whether an entire or meromorphic function can have a wandering domain
   with this property (see \cite[Problem 2.87]{haymanlist} and 
   \cite[Question 8]{waltermero}). 

  We also remark that the assumption on wandering domains 
   in Theorem \ref{thm:nonrecurrent} can be weakened somewhat.
   E.g.\ it is enough to assume that the grand orbit of every wandering
   domain contains at most one singular orbit 
   (see Theorem \ref{thm:nonrecurrent2} for details).
\end{remark}

There are two main ingredients in the proof of this theorem. The
 first is a generalization of Ma\~n{\'e}'s theorem 
 \cite{mane} regarding
 pullbacks of disks that are disjoint from the $\omega$-limit sets
 of recurrent critical values. 
 In particular, we prove the following.

\begin{thm}[Hyperbolic sets] \label{thm:manemain}
 Let $f:\C\to\Ch$ be a nonlinear and nonconstant meromorphic function,
  and suppose that $\P(f)\cap\C$ is bounded. 

 Furthermore, let $K\subset\C$ 
   be a compact subset of the Julia set $J(f)$ with the following
  properties: 
  \begin{enumerate}
    \item $K$ is forward-invariant, i.e.\ $f(K)\subset K$;
    \item $K$ contains neither critical points nor 
               parabolic periodic points;
    \item $K$ does not intersect the $\omega$-limit set of any
       recurrent critical points, or of any singular values
       contained in wandering domains.
  \end{enumerate}
 Then $K$ is a hyperbolic set.
\end{thm}
\begin{remark}[Remark 1]
  Recall that a \emph{hyperbolic set} is a compact, forward invariant
  set $K\subset\C$ such that, for some $\eta>1$ and some 
  $k\in\N$, we have $|(f^k)'(z)|>\eta$ for all $z\in K$.
\end{remark}
\begin{remark}[Remark 2] 
 It is clearly necessary to make some restrictions about 
  unbounded singular orbits; consider e.g.\ the case of an entire
  function that
  has a non-linearizable irrationally indifferent fixed point but 
  no critical points.  
  However, the assumption
  that $\P(f)\cap\C$ is bounded is stronger that what we need;
  we use it merely for convenience of statement.
  Compare Theorem \ref{thm:maneexplicit}.
\end{remark}

The second main result concerns the absence of invariant line fields
 for a large class of ``dynamically tame'' transcendental 
 meromorphic functions. This theorem generalizes a result of McMullen
  \cite[Theorem 3.17]{mcmullenrenormalization}.
 It 
 will be used in a subsequent
 paper \cite{lassesebastiandensity} to establish density of
 hyperbolicity in parameter spaces of certain \emph{real} entire 
 transcendental functions with finitely many singular values. 
 (This is the original motivation behind our study.) 

To state the result. 
 let us say that $f$ is \emph{measurably transitive} if the 
 radial Julia set $J_r(f)$ (see Section \ref{sec:ergodic}) has
 positive measure. Roughly speaking, this means that it is possible
 to pass from small scales to large scales by a univalent iterate
 near any point in $\C$, which is a natural hypothesis in our setting.
 If $f$ is measurably transitive, then $J(f)=\C$ and 
 $f$ is ergodic on the Julia set; see Theorem
 \ref{thm:ergodic} below.

We note that all functions 
  with $\limsup \dist^{\#}(f^n(z),\P(f))>0$ for a positive
  measure set of $z\in J(f)$ (where $\dist^{\#}$ denotes spherical
  distance) are measurably transitive.

\begin{thm}[Absence of line fields] \label{thm:mainlinefields}
 Let $f:\C\to\Ch$ be a nonconstant, nonlinear meromorphic
  function.

 Suppose that $f$ is measurably transitive, and that $f$ is not
  a Latt\`es map. Then $f$ supports no invariant line field on its
  Julia set.
\end{thm}

Using Theorem \ref{thm:manemain} and a result of
 Bock \cite{bockthesis} (see Theorem \ref{thm:ergodic}), it follows
 that a function satisfying the hypotheses of Theorem
 \ref{thm:nonrecurrent} either is measurably transitive, or 
 almost every orbit accumulates at $\infty$.
 Hence 
 Theorem \ref{thm:nonrecurrent} will be proved
 in Section \ref{sec:proofs} by showing that the latter set
 (and, in particular, the set of \emph{escaping points}, which converge
  to infinity under iteration)
 also does not support invariant line fields.

 Graczyk, Kotus and \'{S}wi\k{a}tek
  \cite{graczykkotusswiatek} proved Theorem
  \ref{thm:manemain} under the additional assumptions that
  $\P(f)$ contains no critical points and that $J(f)=\C$.
  They also proved Theorem \ref{thm:nonrecurrent}
  for such functions, assuming that $\infty\notin\P(f)$ as well as
  an additional technical condition. 

Our proof of Theorem \ref{thm:manemain} follows quite closely the account
 of Ma\~n{\'e}'s theorem given by Tan Lei and Shishikura
 \cite{tanleishishikura}, although additional care is required to deal
 with the transcendental case. Likewise, the proof of
 Theorem \ref{thm:mainlinefields} follows McMullen's general strategy,
 but contains also some other ingredients: we use Nevanlinna's theorem on
 completely branched values, and also need to
 develop an argument to deal
 with the singularity at infinity for a transcendental meromorphic function. 
 Martin and Mayer \cite{martinmayer} have given an alternative 
 proof of Theorem \ref{thm:mainlinefields} for rational functions that can
 also be applied to transcendental entire functions. 
 This argument
 is extended in \cite{mayerrempe} to give an alternative proof of 
 Theorem~\ref{thm:mainlinefields} that also applies to Epstein's more
 general class of ``Ahlfors islands maps''.

We should mention that there are a number of other results regarding
 the absence of
 invariant line fields for transcendental entire and meromorphic
 functions which go in a somewhat different direction. In particular,
 in \cite{urbanskizdunik4}, the absence of invariant line fields is
 shown for exponential maps that satisfy a type of Collet-Eckmann
 condition. The article \cite{boettcher} establishes that an invariant
 line field for any meromorphic function cannot be supported on the
 set of points that escape to infinity through logarithmic singularities.
 (This result will be used in the proof of Theorem \ref{thm:nonrecurrent}
 above.) 

Also, our version of Ma\~n\'e's theorem has connections to
  Kisaka's study \cite{kisakasemihyperbolic} of \emph{semi-hyperbolicity} 
 (which requires that \emph{all} pullbacks of sufficiently small disks
  have only a bounded amount
  of branching, which e.g.\ excludes the existence of asymptotic values);
  see also \cite{waltermorosawa}. Okuyama 
  \cite[Theorem 5.1]{okuyamalinearization} has a certain version of 
  Ma{\~n}\'e's theorem for 
  functions with finitely many critical points and asymptotic values,
  while Kotus and Urba\'nski \cite{kotusurbanskielliptic} have previously
  treated the case of elliptic functions.

\subsection*{Structure of the article}
 We begin by proving our version of
  Ma{\~n}\'e's theorem in Section \ref{sec:mane}. 
  Sections \ref{sec:ergodic} to \ref{sec:exceptional} discuss material
  that is not entirely new, but for which we know of no reference that
  presents it as we require. More precisely, Section \ref{sec:ergodic} 
  discusses radial Julia sets for general transcendental entire or
  meromorphic functions, which are important in the
  study of measurable dynamics of such functions. We also
  prove two basic results regarding these sets.  
  Section \ref{sec:linefields} reviews 
  facts about (univalent) line fields.
  In Section \ref{sec:exceptional}, we introduce and discuss the 
  concept of ``branched exceptional'' values, which may be
  of independent interest and does not seem to appear explicitly
  in the literature. The main work for Theorem
  \ref{thm:mainlinefields} is done in Section
  \ref{sec:univalent}, where we classify all meromorphic functions
  that support univalent line fields near their Julia sets. We remark
  that the discussion of Ma\~n\'e's theorem and radial Julia sets in
  Sections \ref{sec:mane} and \ref{sec:ergodic}, on the one hand, and
  of invariant line fields in 
  Sections \ref{sec:linefields} to
  \ref{sec:univalent}, on the other, can be read quite independently from
  each other. 
  Finally, we tie
  up the loose strings in Section \ref{sec:proofs}.

\subsection*{Basic notation}
 We denote the complex plane by $\C$ and the
  Riemann sphere by $\Ch = \C\cup\{\infty\}$. We use a number of
  different notions of distance in this article: Euclidean distance is
  denoted $\dist$, spherical distance is denoted
  $\dist^{\#}$, and hyperbolic distance in some open set $U\subset\C$ is
  denoted $\dist_U$. We use the analogous notation for diameters;
  e.g.\ $\diam_U(X)$ denotes the diameter of $X$ in the hyperbolic 
  metric of $U$. The euclidean disk of radius 
  $\eps$ around $z$ is denoted $\D_{\eps}(z)$, while the
  corresponding spherical disk is $\D^{\#}_{\eps}(z)$.

 Throughout this article $f:\C\to\Ch$ will be a nonconstant,
  nonlinear meromorphic function, unless explicitly stated otherwise. 
  As usual, 
  the Julia and Fatou sets of $f$ are
  denoted by $J(f)$ and $F(f)$, respectively. If $z\in\C$ and
  $U\subset\Ch$ with $f^n(z)\in U$, then the
  \emph{pullback} of $U$ along the orbit of $z$ is the 
  component $W$ of $f^{-n}(U)$ containing $z$. This pullback is called
  \emph{univalent} if $f:W\to U$ is univalent. 
  The \emph{$\omega$-limit set} of a point $z\in\C$ is 
  the
  accumulation set of the forward orbit of $z$ under $f$. 

 We frequently deal with open covers of a
  compact set $K\subset \C$.
  The \emph{Lebesgue covering number} of such a 
  covering is
  the largest number $\delta$ such that the disk of radius $\delta$ around
  any point of $K$ is 
  contained in some element of the open cover in question
  \cite[Theorem 26]{kelleytopology}. 

 We use the following notation to conclude proofs and results: 
  $\proofsymbol$ denotes the end of a proof, $\noproofsymbol$ indicates
  a result cited without proof, and $\subproofsymbol$ completes the proof
  of an auxiliary step within a larger argument. 

\subsection*{Acknowledgments}
  We thank Walter Bergweiler, Adam Epstein, Janina Kotus, Volker Mayer,
   Phil Rippon and Mariusz Urba\'nski for interesting
   discussions.

\section{Ma\~n{\'e}'s theorem} \label{sec:mane}

 In this section, we do the main work for the proof of Theorem 
  \ref{thm:manemain}. More precisely, we establish the ``point version''
  of this theorem (Theorem \ref{thm:manegeneral}). In the rational case,
  this 
  states that a sufficiently small disk around any point of the compact set
  $K$ can be pulled back along arbitrary backwards orbits with only a 
  bounded amount of branching. The 
  ``compact set version'' of our theorem,
  stated in the introduction, is deduced from the point version 
  in
  the following section, as a special case of an observation about 
  radial Julia sets.

 To motivate what follows, we recall that the proof of Ma\~n\'e's 
  theorem essentially relies on the following two principles
  (compare \cite[Lemmas 1 and 2]{mane}).
  \begin{enumerate}
   \item[(I)] Suppose we are given a sequence 
   $V_0\leftarrow V_1 \leftarrow V_2 \leftarrow \dots \leftarrow V_n$ of 
   pullbacks of a 
   disk $V_0$ intersecting the Julia set but not intersecting any 
   recurrent 
   critical orbits. 
   Suppose furthermore that the $V_j$ are ``small'' for $j<n$. 
   Then this pullback passes through
   every non-recurrent critical point at most once, and hence the
   degree of such a pullback is at most the product over the degrees of
   non-recurrent critical points.
  \item[(II)] If a larger disk $\tilde{V_0}$ can be pulled back along the same
   orbit as $V_0$ with a bounded amount of branching, then the 
   pullback of $V_0$ will be ``small''.
  \end{enumerate}
 In the rational case, both these facts are true for any reasonable
  meaning of the word ``small'', e.g.\ in the spherical metric,
  but in the transcendental case, where the
  domain of definition of $f$ is noncompact, things
  are more complicated. 
  For example, if we interpret ``small'' in the Euclidean metric,
  then there is no reason for (II) to be  true, and indeed it can
  be checked that (II) fails e.g.\ for the functions $f(z)=1/\cos(\sqrt{z})$. 
  (Consider preimages of a disk $D_{r/2}(r)$,
   where $r>0$ is sufficiently small.) On the other hand, with respect to the
   spherical metric, there 
   are 
   a number of reasons for 
   (I) to fail, for example if the pullback passes 
   through transcendental singularities, critical points of arbitrarily high
   period, etc. 

  One might think that it is not necessary to consider
   unbounded pullbacks if we are only interested in bounded backward
   orbits. However, this is not true. It is useful to keep in
   mind the case of an exponential map $z\mapsto \exp(z)+\kappa$
   with a bounded Siegel disk $U$. In this case, $K:=\partial U$ is a compact,
   forward invariant set, but $K$ is not a hyperbolic set. Hence any proof 
   of Ma\'n\'e's theorem will break down in this case, precisely
   because it may become necessary to follow pullbacks along the (unbounded)
   singular orbit. In fact, our results imply that the singular value
   is recurrent in this setting; compare Corollary \ref{cor:certainsiegel} 
   below. 

  The main insight we
   need in order to deal with these issues effectively 
   is due to Tan Lei and Shishikura
   \cite{tanleishishikura}: they interpret ``small'' above as
   having bounded diameter in the hyperbolic metric of some suitable
   backwards-invariant open set $\Omega$. Then (II) 
   is 
   immediate (see Lemma \ref{lem:diam}), 
   and (I) can be dealt with (given suitable hypotheses) by
   only considering pullbacks that follow the postsingular set, as in the
   following definition.

 \begin{defn}[Regular points] \label{defn:regularity}
  We say that a point $z_0\in \C$ is \emph{regular} if there are
   $\delta>0$ and $\Delta\in\N$ with the following property:
   if $n\in\N$ and $U$ is a connected component of 
   $f^{-n}(\D_{\delta}(z_0))$ with $U\cap \P(f)\neq\emptyset$, then $U$
   is simply connected and $f^n:U\to \D_{\delta}(z_0)$ is a 
   branched covering of degree at most $\Delta$. 
   (If we want to be more specific,
   we also say that $z_0$ is \emph{$\Delta$-regular}.)
\end{defn}
 \begin{remark}
  This differs from the definition of regularity 
   in \cite{graczykkotusswiatek}, where
   it was required that the map is univalent; i.e. $\Delta=1$.
 \end{remark}

 Our goal now is to prove regularity of $z_0$ under suitable assumptions
  (e.g., when $z_0$ belongs to a set $K$ as in the statement of
   Theorem \ref{thm:manemain}).   
   We follow the ideas of \cite{tanleishishikura} quite closely,
   although our presentation differs somewhat.
   In particular, we use the following
   simple lemma from \cite{tanleishishikura}.

 \begin{lem}[Preimages under maps of bounded degree]
    \label{lem:diam}
  Let $N\in\N$, and let $0<\eta<1$. Then there is a constant
   $C(N,\eta)$ with the following property. Let $z_0\in\C$ and $\delta>0$.
   Suppose that $U\subset\C$
   is simply connected, and let 
    \[ g: U \to \D_{\delta}(z_0)\]
   be a proper map of degree at most $N$. Then every component
   of $g^{-1}(\D_{\eta\cdot\delta}(z_0))$ 
   has diameter at most $C(N,\eta)$ in the
   hyperbolic metric of $U$. 

  As $\eta\to 0$ for fixed $N$, the constant $C(N,\eta)$ also tends to $0$. 
    \qedd
 \end{lem}

\subsection*{An abstract version of the theorem}
To clarify the structure of the proof, we begin by stating explicitly
 the hypotheses
 on the domain $\Omega$ that are needed in the proof (essentially, these
 are the requirements to make (I) above work; the reader will quickly 
 recognize (*) below as a type of nonrecurrence condition). We then deduce
 a version of 
 Ma\~n\'e's theorem under these assumptions. Afterwards, 
 we formulate some
 explicit situations in which our setup applies. 
 For the remainder of this section, let $N_0\in\N$ be the smallest 
 number such that $\D$ can be covered by $N_0$ disks of radius $1/3$
 with centers in $\D$. Recall our standing assumption that
 $f:\C\to\Ch$ is a meromorphic, nonconstant and nonlinear function. 

 \begin{defn}[Regular sets] \label{defn:manesetup}
  Let $S_0$ be a finite set of critical values
   of $f$ and let $D\in\N$ be an integer. We set
   \begin{equation}
     N := D^{\# S_0} \quad\text{and}\quad C_0 := N_0\cdot N\cdot C(N,2/3).
     \label{eqn:N}
   \end{equation}

  A non-empty open set $\Omega\subset\C$ with 
   $f^{-1}(\Omega)\subset\Omega$ is called a
   \emph{$(S_0,D)$-{\controlledset}} if the following condition holds. 
    \begin{enumerate}
      \item[(*)] 
       Let $V\subset\Omega$ be an arbitrary simply connected domain, and let
        $U$ be a component of $f^{-1}(V)$. Assume
        that there is some $W\subset\Omega$ with $V\subset W$ and 
        $\diam_{\Omega}(W)\leq 2C_0$ such that the component
        of $f^{-1}(W)$ containing $U$ intersects $\P(f)$.

       Then 
        $f:U\to V$ is a proper map of degree at most $D$ with at most
        one branched value. (In particular, $U$ is simply connected.)
        Furthermore, if there is a branched value $s$ of this map, then
        $s\in S_0$ and $f^j(s)\notin U$ for all $j\geq 0$.
    \end{enumerate}

  (A set that is $(S_0,D)$-regular for some $S_0$ and $D$ is 
   called a \emph{\controlledset}.)
 \end{defn}
 \begin{remark}
   The assumptions of (*) are satisfied, in particular, if
    $\diam_{\Omega}(V)\leq 2C_0$ and $U\cap \P(f)\neq\emptyset$. 
    The slightly weaker assumption in (*)
    is introduced to allow for pullbacks that leave the postsingular set,
    but in a sense remain ``close'' to it. 
 \end{remark}
 
 Given an $(S_0,D)$-{\controlledset} $\Omega$ and a simply connected
  domain $V\subset\Omega$, let us say that a component $U$ of
  $f^{-1}(V)$ is a ``\good'' preimage component of $V$ if the
  assumption in (*) is satisfied. (This notion depends on $S_0$, $D$
  and $\Omega$. Whenever we use the term, it will be clear from
  the context what these are.) We remark that, by definition, only sets 
  $V$ with
  $\diam_{\Omega}(V)\leq 2C_0$ can have {\good} preimage components.

In this situation, principle (I) can be phrased as the following simple lemma.
\begin{lem}[Long {\good} pullbacks have bounded branching]
 Suppose that $\Omega$ is an $(S_0,D)$-{\controlledset}
     and that $V_0\subset \Omega$ is simply connected. Consider
   a pullback $V_0\leftarrow V_1 \leftarrow \dots \leftarrow V_n$ of
   $V_0$; i.e., $V_{j+1}$ is a component of $f^{-1}(V_j)$. 
   Assume that, for each $j<n$ with $V_j\cap \P(f)\neq\emptyset$, 
   $V_{j+1}$ is a {\good} preimage component of $V_j$. 

 Then all $V_j$ are simply connected, and $f^n:V_n\to V_0$ is a proper map
  of degree at most $N$, where $N$ is as in 
  (\ref{eqn:N}). \label{lem:deg}
\end{lem}
\begin{proof}
 We first prove by induction that each 
  $V_{j+1}$ is simply connected, and that
  $f:V_{j+1}\to V_j$ is a proper map, branched of degree at most $D$ 
  at most over a single point
  of $S_0$. Indeed, suppose that we have already shown that
  $V_j$ is simply connected. If $V_j\cap \P(f)=\emptyset$, 
  then $f:V_{j+1}\to V_j$ is a conformal isomorphism, and $V_{j+1}$ is
  simply connected. Otherwise, we can apply (*) to see
  that $f:V_{j+1}\to V_j$ has degree at most $D$ and is 
  branched at most over a single point of $S_0$. This implies
  that $V_{j+1}$ is simply connected.

 Furthermore, let $s\in S_0$. Then, by 
  the final statement in (*), there is at most one $j$ 
  such that $f:V_{j+1}\to V_j$ is branched over $s$. Hence it follows that
  $\deg(f^n:V_n\to V_0)\leq D^{\# S_0} = N$, as claimed.
\end{proof}

 \begin{thm}[General form of Ma\~n\'e's theorem] \label{thm:manegeneral}
  Suppose that 
  $\Omega$ is an $(N_0,D)$-{\controlledset} and let $N$ and $C_0$
  be as in (\ref{eqn:N}). Then every point of $\Omega$ is
  $N$-regular.

 More precisely, let $z_0\in\Omega$ and let
   $\delta>0$ be sufficiently small to ensure
   that $\D_{2\delta}(z_0)\subset\Omega$ and
   $\diam_{\Omega}(\D_{2\delta}(z_0))\leq C_0$. Set 
   $V_0 := \D_{\delta}(z_0)$. 

  Let $n\geq 0$, and suppose that $V$ is a component of 
   $f^{-n}(V_0)$
   with $V\cap \P(f)\neq\emptyset$. Then
  \begin{enumerate}
   \item $V$ is simply connected and $\deg(f^n:V \to D_0)\leq N$, and%
      \label{item:deg}
   \item $\diam_{\Omega}(V)\leq C_0$.  \label{item:diam}
  \end{enumerate}
 \end{thm}
 \begin{proof}
  The proof is by induction. Both claims are trivial for
   $n=0$.

  Let $V_0 \leftarrow V_{1} \leftarrow \dots \leftarrow V_n= V$ 
   be the pullback of $V_0$ along the orbit of $V$.
   By part (\ref{item:diam}) of 
   the induction hypothesis, we have
   $\diam_{\Omega}(V_k)\leq C_0$ for all $k<n$. Also 
   $V_{k+1}\cap \P(f)\neq \emptyset$. Hence each
   $V_{k+1}$ is a {\good} preimage component of $V_k$, and 
   (\ref{item:deg}) follows from the induction hypothesis and
   Lemma \ref{lem:deg}. 

\medskip

  So, let us prove (\ref{item:diam}). To do so, 
   cover the disk $V_0$ by
   $N_0$ disks of radius $\delta/3$ with centers in $V_0$. 
   Let $U^1,\dots,U^{\ell}$ be the collection of those preimage components
   of any of these disks that intersect $V$. So each
   $U^j$ is a component of $f^{-n}(D^j)$, where $D^j=\D_{\delta/3}(z^j)$ with
   $z^j\in V_0$, and the $U^j$ cover $V$.
   By (\ref{item:deg}), the number $\ell$ does not exceed
   $N_0\cdot N$. We claim that each $U^j$ has diameter at most
   $C(N,2/3)$ in the hyperbolic metric of $\Omega$. 

\smallskip

 To show this, let us fix $j$ for the moment, and write
   $U := U^{j}$. Consider the
   disk $\tilde{D} := \tilde{D}_j := \D_{\delta/2}(z^j)$. Observe that
   $\D_{\delta}(z^j) \subset \D_{2\delta}(z_0)$
   by construction, so 
   we will be able to apply the induction hypothesis to pullbacks of
   the disk
   $\tilde{D}$. Let
    \[ \tilde{D} =: \tilde{U}_0  \leftarrow
         \tilde{U}_1 \leftarrow \dots \leftarrow 
         \tilde{U}_n  \]
   be the pullback of $\tilde{D}$ along the orbit of $U$. Note that 
   $V_k \cap \tilde{U}_k \neq\emptyset$ for all $k$ by construction. 

 If $k<n$ is such that
   $\tilde{U}_k\cap \P(f)\neq\emptyset$, then 
   $\diam(\tilde{U}_k)\leq C_0$ by part (\ref{item:diam}) of
   the induction hypothesis. So the
   set $W := \tilde{U}_k\cup V_k$ has diameter at most $2C_0$ in
   $\Omega$. The component of
   $f^{-1}(W)$ containing $\tilde{U}_{k+1}$ 
   also contains $V_{k+1}$, and hence intersects the postsingular set.    
   So $\tilde{U}_{k+1}$ is a {\good} preimage component of 
   $\tilde{U}_k$, and  Lemma \ref{lem:deg} implies that 
   $\tilde{U}_n$ is simply connected and 
      \[ \deg(f^{n}:\tilde{U}_n \to \tilde{D}) \leq N. \]
    By Lemma \ref{lem:diam}, it follows that
    $\diam_{\Omega}(U)\leq \diam_{\tilde{U}_n}(U) \leq C(N,2/3)$, 
    as claimed.

\smallskip
  Hence
     \begin{align*}
  \diam_{\Omega}(V) &\leq 
      \diam_{\Omega}\left(\bigcup_{j=1}^{\ell} \cl{U^j}\right) \leq
            \sum_{j=1}^{\ell}
         \diam_{\Omega}(U^j) \\ 
      &\leq  \ell\cdot C(N,2/3) \leq
              N\cdot N_0 \cdot C(N,2/3) =
                 C_0.
     \end{align*}
   The proof of the induction step is complete. 
 \end{proof}

\subsection*{An explicit version of Ma{\~n}\'e's theorem}

 We will now discuss  under which hypotheses
  it is possible to construct a regular set $\Omega$. 
  To do so, we introduce the following definitions.

 \begin{defn}[Recurrence]
  Let $s\in S(f)\cap\C$, 
   and let $V$ be a simply connected neighborhood of $s$.
   We say that a component $U$ of $f^{-1}(V)$ is \emph{unbranched} if 
   $f:U\to V$ is univalent. Otherwise we say that $U$ is \emph{singular}.
   More generally, we say that a component 
   $U$ of $f^{-1}(V)$ is \emph{$d$-controlled} if 
   $f:U\to V$ is proper of degree at most $d$ and has no branched points
   except possibly
   over $s$. 

  The singular value $s$ is \emph{nonrecurrent} if there is some $\delta>0$
   such that no singular
   preimage of $V=\D_{\delta}(s)$ intersects the forward
   orbit of $s$. (Otherwise $s$ is called \emph{recurrent}.)

 We also say that $s$ is
   \emph{strongly nonrecurrent} if furthermore there is $d_s\in\N$ such that
   every component of $f^{-1}(V)$ that intersects $\P(f)$
   is $d_s$-controlled. (Otherwise, we call $s$ \emph{almost recurrent}.)
 \end{defn}
 \begin{remark}
  If $\P(f)\cap\C$ is bounded, then 
   there are at most finitely many 
   recurrent singular values, and every nonrecurrent singular value
   is strongly nonrecurrent. 
 \end{remark}

  For $z_0\in J(f)\cap\C$ and $\delta>0$, let us define
   \[ 
    S_{\delta}(z_0) := 
       \cl{\left\{s\in S(f): \exists j\geq 0: |f^j(s)-z_0|\leq \delta\right\}}. 
   \]
 Let us also say that a wandering domain $G$ is ``bad'' if $f$ is not a 
 covering map from $G$ to the Fatou component containing $f(G)$ and 
 furthermore
 $G\cap \P(f)\neq\emptyset$. Note that the grand orbit of a bad wandering
 domain must contain at least two different singular values. 

\begin{thm}[Explicit hypotheses for Ma{\~n}\'e's Theorem]
   \label{thm:maneexplicit}
  Suppose that 
   $z_0\in J(f)\cap\C$ is not a parabolic periodic point. Also assume that
   $z_0$ is not a limit point of the forward orbit of a 
   bad wandering domain.

 Suppose furthermore that for some $\delta>0$, the set
  $S_1 := S_{\delta}(z_0)$ is bounded, and 
  that every $s\in S_1$ is strongly nonrecurrent. 

 Then there exist a finite subset $S_0\subset S_1$, a number
   $D\in\N$ and a
    $(S_0,D)$-{\controlledset} $\Omega\ni z_0$. 
  In particular, $z_0$ is regular by Theorem \ref{thm:manegeneral}.
\end{thm}
\begin{remark}
 It is not strictly necessary to assume that $S_1$ is bounded; we could allow
  unbounded singular sets
  if we were willing to introduce additional technical
  assumptions. 

 Likewise, we could still relax the definition of bad wandering domains
  somewhat, e.g.\ allowing components that intersect the postsingular set
  to be mapped as a branched covering with
  at most one branched point and bounded degree.

 However, it seems to us that the extra generality obtained
  would not justify
  the additional technicality of assumptions.
\end{remark}
\begin{proof}
  Let $s\in S_1$. Then there are $\delta_s>0$ and $d_s\in\N$ such that 
   every singular preimage component of $\D_{\delta_s}(s)$ that intersects 
   $\P(f)$
   is $d_s$-controlled and disjoint from the forward orbit of $s$. 

  Since $S_1$ is compact, we can pick a finite set $S_0\subset S_1$
   such that the collection of disks $\D_{\delta_s}(s)$ with $s\in S_0$
   covers $S_1$. Let $\delta$ be the Lebesgue covering number of this
   covering.

  Now let $D := \max_{s\in S_0} d_s$, and set 
   $N := D^{\# S_0}$ and $C_0 := N_0\cdot N \cdot C(N,2/3)$, as in 
   (\ref{eqn:N}). 

  To construct the domain $\Omega$, let 
   $\Orbit$ be a repelling periodic orbit of period at least $3$ that
   is not contained in the orbit of $z_0$. Then the backward orbit of 
   $\Orbit$ is dense in the Julia set by the Picard and Montel theorems
   (compare Section \ref{sec:exceptional}). 

  Consequently, if we choose $\eps>0$ small enough  and
   $M>0$ sufficiently large, then the domain 
   $\Omega_1 := \C\setminus f^{-M}(\Orbit)$ will have the following
   properties. 
   If $s\in S_1$ with
   $\dist(s,J(f)) < \eps$, then 
   \begin{equation}
      \{z: \dist_{\Omega_1}(z,s) \leq 2C_0 \} \subset
            \D_{\delta}(s); \label{eqn:closetojulia}
   \end{equation}
  while for $s\in S_1$ with $\dist(s,J(f)) \geq \eps$,
   \begin{equation}
      \{z: \dist_{\Omega_1}(z,s) \leq 2C_0 \} \subset
             F(f). \label{eqn:farfromjulia}
   \end{equation}
 (This follows readily from standard estimates on the hyperbolic metric;
  compare e.g.\ \cite[Section 2.2]{mcmullenrenormalization}.) 

 Now consider the subset $S_2\subset S_1$ consisting of all
  singular values $s\in S_1$ with $\dist(s,J(f))\geq \eps$ that are
  not contained in a good (i.e., not bad) wandering domain. 
  Then $S_2$ is a compact subset of the Fatou set, and hence is contained in
  finitely many Fatou components $U_1,\dots ,U_n$. 
  Each of these Fatou components is one of the following.
   \begin{itemize}
     \item An attracting domain, a Siegel disk or a Herman ring, or 
       an iterated preimage of such a domain. In each case, the orbit of
       $S_2\cap U_j$ is compactly contained in the Fatou set.
     \item A parabolic domain, 
        an iterated
        preimage of such a domain,
        or a \emph{bad} wandering domain. In this case, the forward orbit of 
        $S_2\cap U_j$ may accumulate
        at a parabolic point, at infinity or at some other points of the 
        Julia 
        set. However, it does not
        accumulate on $z_0$ by assumption.
     \item A \emph{Baker domain}, or an iterated preimage of such a domain.
        I.e., $f$ is a transcendental meromorphic function, $U_j$ is
        a (pre-)periodic Fatou component and
        $\infty$ is a limit function of the sequence $f^n|_{U_j}$. 
        We claim that the forward orbit of $S_2\cap U_j$ cannot accumulate
        at $z_0$ in this case either. 

       To show this, let $V_0\to V_1 \to \dots \to V_p = V_0$ be the periodic
        orbit of Fatou components to which $U_j$ eventually maps, 
        and let $a_i\in\Ch$ be the limit function of
        the sequence
        $f^{pn}|_{V_i}$. Then $a_i=\infty$ for at least one $i$.
        Furthermore, if $a_i\in\C$, then $a_{i+1}=f(a_i)$, and if
        $a_{i}=\infty$, then $a_{i+1}$ is an asymptotic value of $f$;
        more precisely, there is a curve to infinity in $V_i$ whose
        image is a curve in $V_{i+1}$ ending at $a_{i+1}$.
        (See \cite[Theorem 13]{waltermero}.) 

       Since the cycle of Baker domains contains points of the postsingular
        set by assumption, such an asymptotic values $a_{i+1}$ cannot be
        strongly
        nonrecurrent. Indeed, if we take
        a preimage of a disk around $a_{i+1}$ that contains a tail of
        the aforementioned curve, then this preimage component
        is not mapped
        as a branched covering, but contains postsingular points.

       To summarize, every finite limit point of the sequence
        $f^n|_{U_j}$ lies in the forward orbit of some 
        asymptotic value that is recurrent or almost recurrent. By assumption,
        $z_0$ cannot be one of these limit points. 
   \end{itemize}
   It follows that
     \[ z_0 \notin A := \cl{\bigcup_{j\geq 0} f^j(S_2)}. \]

   We now set
     \[ \Omega := \Omega_1 \setminus \left( A \cup
                    \cl{\bigcup_{j\geq 0} f^j(S(f)\setminus S_1)}\right). \]
    Then $z_0\in \Omega$; since $\Omega$ is the complement of a forward
     invariant set, we also have $f^{-1}(\Omega)\subset\Omega$. 
     It remains to establish the condition (*) from
     Definition \ref{defn:manesetup}.

   Suppose that $U,V,W\subset\Omega$ are as in (*). 
    If $V\cap S(f) = \emptyset$, then
   (since $V$ is simply connected) $f:U\to V$ is a conformal isomorphism.
    Otherwise, $V$ contains some singular value $s$; by construction we
    must have $s\in S_1\setminus S_2$.

   If $\dist(s,J(f))<\eps$, then (\ref{eqn:closetojulia}) and the fact
    that 
    $\diam_{\Omega}(W)\leq 2C_0$ imply that
    $W\subset \D_{\delta}(s)\subset \D_{\delta_{s_0}}(s_0)$ for some
    $s_0\in S_0$. By assumption, the component of 
    $f^{-1}(\D_{\delta_{s_0}}(s_0))$ containing $U$ intersects the
    postsingular set. If this component is unbranched, then
    $f:U\to V$ is univalent. Otherwise, by construction
    the map $f:U\to V$ is branched only over $s_0$, and of degree at most $D$,
    as claimed.

   On the other hand, if $\dist(s,J(f))\geq\eps$, then by definition of
    $S_2$, the singular value
    $s$ is contained in a wandering domain $G$, and this domain
    is not bad. By (\ref{eqn:farfromjulia}), we then have
    $W\subset G$. So the component of 
    $f^{-1}(G)$ containing $U$ intersects the postsingular set.
    By the definition of bad wandering domains, this component is mapped
    onto $G$ as a holomorphic covering map. Since $V$ is simply connected,
    it follows that $f:U\to V$ is
    univalent.

   This completes the proof of the theorem.
\end{proof}

The following standard lemma will be used to apply 
 Theorem \ref{thm:maneexplicit}. 

\begin{lem}[Bounded pullbacks] \label{lem:pullbacks}
 Suppose that $z_0\in\C$ is a regular point, and let $\eps>0$ and $R>0$. 
  Then there are
  $\delta>0$ and $D\in\N$ with the following property:

 If $n\in\N$ and $V$ is a component
  of $f^{-n}(\D_{\delta}(z_0))$ with 
       \[ f^j(V)\cap \D_R(0) \neq \emptyset \]
   for $j=0,\dots , n$, then $\diam(V)\leq \eps$ and 
     \[ \deg(f:V\to \D_{\delta}(z_0)) \leq D. \]
\end{lem}   
 \begin{proof}
   Let $\Delta$ and $\delta_0$ 
    be the constants from the regularity of $z_0$, and let
    $d$ be the largest degree of a critical point $c\in 
    \cl{\D_R(0)}\setminus \P(f)$. We set $D := d\cdot \Delta$. 

   Pick $\rho>0$ such that, for every $w\in\cl{\D_{R}(0)}$,
    every component $U$ of $f^{-1}(\D_{\rho}(w))$ with
    $U\cap\D_{R}(0)\neq\emptyset$ 
   \begin{itemize}
     \item is simply connected, 
     \item is mapped as a proper map
            by $f$  
     \item contains at most one critical point of $f$, and no 
        critical points that are outside of $\cl{\D_{R}(0)}$. 
   \end{itemize}

   By assumption, we know that every component of 
    $f^{-n}(\D_{\delta_0}(z_0))$ that intersects $\P(f)$ is simply
    connected and mapped by $f^n$ with degree at most $\Delta$. 
    By Lemma \ref{lem:diam}, it follows that, for sufficiently small
    $\delta_1<\delta_0$, every component $U$ of 
    $f^{-n}(\D_{\delta_1}(z_0))$ that intersects $\P(f)$ and
    $\D_{R}(0)$ has diameter
    at most $\rho$. 

   Now let $\tilde{V}$ be a component 
    of $f^{-n}(\D_{\delta_1}(z_0))$ with 
    $f^j(\tilde{V})\cap \D_R(0)\neq\emptyset$ for $j=0,\dots,n$. 
    Then it follows from the above that $\tilde{V}$ is simply connected and
       \[ \deg(f^n:\tilde{V} \to \D_{\delta_1}(z_0)) \leq d\cdot \Delta = D.\]
    Using Lemma \ref{lem:diam} again, we can pick a sufficiently small
     $\delta<\delta_1$ (independent of $n$)
     such that also $\diam(V)\leq \eps$. 
 \end{proof} 

\subsection*{Applications to boundaries of Siegel disks}

\begin{cor}[Recurrent singular values and Siegel disks]
   \label{cor:siegelstrongnonrecurrence}
 Let $f:\C\to\C$ be a nonconstant, nonlinear meromorphic function with
  finitely many singular values, and
  let $U$ be a Siegel disk of $f$.

 Then $\partial U$ is contained in the limit set of
  a recurrent or almost recurrent
  singular value of $f$. 
\end{cor}
\begin{remark}
 An analogous statement holds for Herman rings, where the boundary will
  be contained in the limit sets of one or two recurrent or almost
  recurrent 
  singular values. 
\end{remark}
\begin{proof}
 Suppose that there is some $z_0\in \partial U$ 
  not contained in the limit set
  of any recurrent or almost recurrent
  singular value  and 
  that $z_0$ is not a parabolic periodic point. 
  Since functions with a finite set of singular values have no wandering
  domains \cite{bakerkotuslu_classS}, we are in a position to apply Theorem
  \ref{thm:maneexplicit}. It follows that there is a small disk 
  $D$ around $z_0$ such that any pullback of $D$ along the postsingular 
  set --- and, in particular, along $\partial D$ ---  undergoes at most a 
  finite amount of branching. Using standard distortion estimates
  \cite[Lemma 2.2]{juliajohn} and the fact that $z_0$ is in the Julia
  set, we see easily that 
  the spherical diameter of such pullbacks shrinks to zero
  along any backward orbit in $\partial D$, which contradicts the fact that
  $D$ intersects the Siegel disk $U$. 

 It follows that
  $\partial U$ is contained in the union of the limit sets of finitely
  many recurrent or
  almost recurrent singular values. That one of
  these contains the whole boundary follows, as in the rational case,
  from the fact that $f$ is ergodic with respect to harmonic measure
  on $\partial U$. 
\end{proof}

As far as we know, the following corollary is new even for exponential and
 trigonometric functions. 
\begin{cor}[Siegel disks of certain meromorphic functions]%
    \label{cor:certainsiegel}
  Suppose that $f$ is either an exponential map, $f(z) = \exp(z)+\kappa$,
   or a nonconstant, nonlinear meromorphic function with no
   asymptotic values and finitely many
   critical values. Assume also that the degrees of the critical points of
   $f$ are bounded by some constant $\Delta$. 

  Then the boundary of any Siegel disk of $f$ is contained in the
   limit set of some recurrent singular value.
\end{cor}
\begin{proof}
  Under the given assumptions, any nonrecurrent singular value is also
  strongly nonrecurrent, so we can apply the previous corollary. 
\end{proof}

We remark that it would be nice to avoid the requirement of
 \emph{strong} nonrecurrence in Corollary \ref{cor:siegelstrongnonrecurrence}. For example,
 consider the maps $b z \exp(z)+a$, which have
 one asymptotic value at $a$ and one critical value at 
 $v=a-b/e$. Suppose that both singular values are non-recurrent, but
 $v$ accumulates on $a$. In such a situation, our theorem does not
 apply, but it does not seem unreasonable to expect that some
 form of Ma{\~n}\'e's theorem still holds. 

\section{Radial Julia sets} \label{sec:ergodic}
 \begin{defn}[Radial Julia sets] \label{defn:radialjulia}
  Let
   $f:\C\to\Ch$ be a nonconstant and nonlinear meromorphic function. The
   \emph{radial Julia set}
   $J_r(f)$ is the set of all points $z\in J(f)$ 
    with the following property:
    there is some $\delta>0$ such that, for infinitely many $n\in\N$, the 
    spherical 
    disk $\D_{\delta}^{\#}(f^n(z))$ can be pulled back univalently along the
    orbit of $z$.
 \end{defn}
 \begin{remark}[Remark 1]
  If $z\in J(f)$ and 
  $\limsup \dist^{\#}(z,\P(f))>0$, then $z\in J_r(f)$, since a disk that
  does not intersect $\P(f)$ can be pulled back univalently along any
  backward orbit. 
 \end{remark}
 \begin{remark}[Remark 2]
  It is essential to use the spherical metric in the above definition,
   rather than e.g.\ the Euclidean metric. For example, consider the
   function $f(z)=\sin(z)/2$, for which the origin attracts
   both critical values of $f$.
   Setting $\delta := \dist(\P(f),J(f))$, it follows
   that any disk of radius $\delta$ around some point of the Julia set
   can be pulled back univalently along any backward orbit.
   By a result of McMullen \cite{hausdorffmcmullen}, 
   the Julia set $J(f)$ has positive area. 
   However, by Theorem \ref{thm:ergodic},
   the area of the
   radial Julia set $J_r(f)$, as defined above, is zero.
 \end{remark}

\subsection*{Measurably transitive functions}

 It is well-known 
  that $J_r(f)$ has either full or zero Lebesgue measure
  and that, in the former case, almost every orbit is dense in the plane. 
  (For rational functions, this fact can be found in 
      \cite[page 608]{alexmishaanalytictransforms}, while it was
    proved for meromorphic functions by Bock \cite{bockthesis}.) 
  As mentioned in the introduction, we will call $f$
  \emph{measurably transitive} if  $J_r(f)$ 
  has positive (and hence full)
  measure.

 Since  \cite{bockthesis} is not
  widely available, let us sketch a proof of the above-mentioned result, 
  following McMullen's argument \cite[Theorem 3.9]{mcmullenrenormalization} 
  for the rational case. 
  We begin with a preliminary observation. 

 \begin{lemdef}[Pullbacks shrink] \label{lem:pullbacksshrink}
  Let $z\in J_r(f)$. Then there exists a disk 
   $D=\D_{\delta}^{\#}(\zeta)$ such
   that, for infinitely many $n$, the point
   $f^n(z)$ belongs to $D$ and the larger disk 
   $\D_{2\delta}^{\#}(\zeta)$
   pulls back univalently along the orbit of $z$.

  In particular, the pullbacks $D_n\ni z$ of $D$ have diameter tending to
   zero, and $f^n:D_n\to D$ has uniformly bounded distortion.

  (We will refer to $D$ as a \emph{disk of univalence} for $z$.) 
 \end{lemdef}
\begin{proof}
 Let $\delta'$ be the constant from Definition \ref{defn:radialjulia},
  and let $n_k$ be a sequence such that $\D_{\delta'}(f^{n_k}(z))$ 
  can be pulled back
  univalently to $z$. Choosing $\zeta$ to be a limit point of the sequence
  $f^{n_k}(z)$ and letting $\delta < \delta'/2$ proves the first claim. 

 The fact that the
  $f^n$ have uniformly bounded distortion on $D_n$ follows from Koebe's
  distortion theorem \cite[Theorem 1.3]{pommerenke}. 
  If $\diam(D_n)$ did not tend to zero, then
  there would be a disk $U$ around $z$ with $f^n(U)\subset D$ for
  infinitely many $n$, which is impossible since $z\in J(f)$. 
\end{proof} 

\begin{thm}[Ergodicity]  \label{thm:ergodic}
 Suppose that $f$ is measurably transitive; that is,
  $J_r(f)$ has positive measure. Then $J_r(f)$ has full
  measure in $\Ch$ (in particular, $J(f)=\C$), and almost
  every point  
  $z\in\C$ 
  has a dense orbit. Furthermore, any set that is forward
  invariant under $f$ has either full or zero Lebesgue measure; in particular,
  the action of $f$ on $\C$ is ergodic. 
\end{thm}
\begin{proof}
 Notice that $J_r(f)$ is forward invariant. 
  Let $F$ be any forward-invariant set such that $X := F\cap J_r(f)$ 
  has positive
  measure. We will show that $X$ has full Lebesgue measure in the 
  sphere. 

 Let $z$ be a Lebesgue density point of $X$, and let
  $D$ be a disk of univalence for $z$. Then
     \[ \frac{\meas(X\cap D_n)}{\meas(D_n)}\to 1 \]
  (where $D_n$ is the pullback from
  Lemma \ref{lem:pullbacksshrink}). Since the distortion of
  $f^n|_{D_n}$ is bounded independently of $n$, it follows that
    \[ \frac{\meas(D\setminus X)}{\meas(D)} \leq
         K\cdot \frac{\meas(D_n\setminus X)}{\meas(D_n)} \to 0 \]
   for some constant $K$. Hence 
   $\meas(X\cap D)=\meas(D)$; i.e.,  $X$ has
   full measure in $D$. 
   As $X$ is forward-invariant and $D$ is an open disk
   intersecting the Julia set, 
   it follows that $X$ has full Lebesgue measure in the plane.

 In particular, $J_r(f)$ itself
  has full measure, and any forward invariant set
  has either full or zero measure in the plane. Note furthermore
  that the set of points whose orbits never enter a given open set 
  $U\subset\Ch$ is forward invariant. Since this set
  is disjoint from the positive measure set $U$, it must have zero Lebesgue
  measure. Hence almost every point has a dense orbit. 
\end{proof}

\subsection*{Remarks about the definition of the radial Julia set}

Radial Julia sets, as the sets where it is possible to go from small
 to large scales via univalent iterates, were originally introduced
 for rational functions. The concept of measurable transivity
 appears to have made its first (implicit) appearance in
 \cite{mishatypical}, while the radial Julia set was explicitly
 introduced in \cite{urbanskietalconformal,mcmullenradial}. 
 The name ``conical Julia set'' is also sometimes
 used (there are in fact a number
 of different definitions of conical Julia sets; see
 \cite{przytyckiconical} for a discussion and interesting results about their
 relation to each other); the
 term ``radial Julia set'' was coined by McMullen 
 \cite{mcmullenradial} who used it 
 in analogy to his studies of Kleinian groups. 

The radial Julia set plays a fundamental role in the study of measurable
 properties of conformal dynamical systems, particularly in the
 transcendental case, where $J_r(f)$ may frequently be much smaller
 than $J(f)$. E.g.\ $J(f)$ may have positive measure, while $J_r(f)$ has
 Hausdorff dimension strictly less than two, see below. 
 (Work by Avila and Lyubich \cite{avilalyubichfeigenbaum2}
  suggests that this may also happen for rational functions, even
  quadratic polynomials. On the other hand, it follows from recent
  work of Buff and Ch\'eritat \cite{buffcheritatarea} that there are
  also quadratic polynomials for which $J(f)$ has positive measure and
  $J_r(f)$ has Hausdorff dimension $2$.) 
  In the rational case, it is known 
  \cite{urbanskietalconformal} that
  the Hausdorff dimension of $J_r(f)$ coincides with several other
  dynamically 
  important quantities, including Shishikura's \emph{hyperbolic dimension},
  the supremum over the dimensions of all hyperbolic subsets of
  $J(f)$.

In the transcendental setting, the fact that 
 the hyperbolic dimension need not equal the dimension of the 
 full Julia set is implicit already in the work of Stallard 
 \cite{stallardhyperbolicmero}. (Stallard discusses 
 \emph{critical Poincar\'e exponents} rather than the hyperbolic dimension, 
 but it is easy to see that the former is an upper bound for the latter.) 
 A closer investigation of this phenomenon 
 was initiated 
 by Urba\'nski and Zdunik \cite{urbanskizdunik1}, who considered 
 the case where $f$ is a hyperbolic exponential map 
 $f(z)=\lambda \exp(z)$ with a single attracting basin. 
 Here, 
 $J_r(f)$ is $J(f)$ minus the set 
\[ I(f) := \{z\in\C:f^n(z)\to\infty\} \] 
 of escaping points, and Urba\'nski and Zdunik prove a number of 
 fundamental results. In particular, they show that the Hausdorff-dimension
 $h$ of $J_r(f)$ lies strictly between $1$ and $2$ (while 
 $J(f)$ has dimension $2$ by a result of McMullen) and agrees with the
 hyperbolic dimension of $f$.
 They also construct natural
 conformal measures and invariant measures supported on the set $J_r(f)$.   

 In \cite{urbanskizdunik2}, a similar program was carried out for
 the \emph{non-hyperbolic} 
 case of an exponential map whose singular value $0$ escapes to infinity 
 ``sufficiently fast'' under iteration; here $J_r(f)$ is the set of points
 that do not accumulate on the singular orbit. 
 This program has subsequently been adapted to a variety of situations where
 there is good control over the postsingular set, and it seems likely that
 a good understanding of
 the set $J_r(f)$ will be the key to studying the measurable dynamics of
 more general transcendental functions. It is possible to show
 \cite{hypdim} that, 
 as in the rational case, the hyperbolic dimension of $f$
 and the Hausdorff dimension of
 $J_r(f)$ always coincide. In this context, let us note the following
 simple fact, which we will use to 
 deduce Theorem \ref{thm:manemain}
 from the
 results of the previous section.

\begin{lem}[Hyperbolic sets and $J_r(f)$] \label{lem:hyperbolicsets}
 Let $K\subset\C$ be a hyperbolic set. Then $K\subset J_r(f)$. Conversely,
  every compact, forward invariant set $K$ with $K\subset J_r(f)$ is
  a hyperbolic set. 
\end{lem}
\begin{proof}
 Let $K$ be a hyperbolic set. Then we can pick $k\in\N$ and a bounded open
  neighborhood $W$ of $K$ such that $|(f^k)'(z)|>\eta>1$ for all
  $z\in W$. 
  For every $w\in f^k(K)$, pick an  open
  neighborhood 
  $U(w)$ small enough so that every component of $f^{-k}(U(w))$ that
  intersects $K$ is completely contained in $W$. 

 The sets $U(z)$ form an open covering of $f^k(K)$; let $\delta$
  be its Lebesgue covering number.  Now let 
  $z_0\in K$ and $w_0:= f^k(z_0)$, and let $V$ be the component of
  $f^{-k}(\D_{\delta}(w_0))$ containing $z$. Then $V\subset W$ and 
  hence there
  is a branch $\phi:\D_{\delta}(w_0)\to V$ of $f^{-k}$. Since
  $|\phi'(w)|<1/\eta$ for all $w\in\D_{\delta}(w_0)$, it follows that
  $V\subset \D_{\delta}(z_0)$.

 It now follows by induction that $\D_{\delta}(w_0)$ can be pulled back
  univalently along any backward orbit of $w_0$ that is contained in 
  $K$. Hence $K\subset J_r(f)$, as claimed.

\smallskip

 For the converse direction, let $K\subset J_r(f)$ be compact
  and forward invariant. We may assume that
  $\theta := \min_{z\in K} |f'(z)| \leq 1$, as otherwise there
  is nothing to prove.

  Let $z_0\in K$. It follows from Lemma
  \ref{lem:pullbacksshrink} 
  that $\limsup |(f^n)'(z_0)|=\infty$, so we can pick some
  $n(z_0)$ and some open neighborhood $U(z_0)$ of $z_0$ such that
  $|(f^{n(z_0)})'(z)|\geq 2$ for all $z\in U(z_0)$. 

 Since $K$ is compact, it is covered by finitely many such neighborhoods;
   let us call them $U_i$ and denote the corresponding numbers
   $n(z_0)$ by $n_i$. 
   Define $N:= \max_i n_i$, and let $m$ be sufficiently large
   that 
      \[ 2^m > \frac{2}{\theta^{N}}. \]
  Now we set $n := m\cdot N$; we claim that
   $|(f^n)'(z)|>2$ for all $z\in K$. 

 To prove this claim, define sequences
   $j_k$ and $z_k$ inductively by setting $z_0 := z$, choosing 
   $j_k$ such that $z_k \in U_{j_k}$, and defining
   $z_{k+1} := f^{n_{j_k}}(z_k)$. 

  Let $k$ be maximal with
   $p := n - n_{j_0}-n_{j_1}-\dots-n_{j_k} \geq 0$. Then $p<N$ and $k\geq m$ by choice of $n$. 
    We can now write
    \[ f^n(z) = f^p(z_k) = f^p(f^{n_{j_k}}(z_{k-1})) = \dots = 
        f^p(f^{n_{j_{k}}}(f^{n_{j_{k-1}}}(\dots (f^{n_{j_0}}(z))\dots))). \]
   So by the chain rule, 
     \[ |(f^n)'(z)|\geq 
           \theta^p \cdot 2^k \geq \theta^N\cdot 2^m > 2 \]
   by choice of $m$.
\end{proof}

\begin{proof}[Proof of Theorem \ref{thm:manemain}]
 We recall the setting of the theorem: $f$ is a nonconstant, nonlinear
  meromorphic function with $\P(f)\cap\C$ bounded,
  $K\subset\C$ is compact and forward invariant 
  and contains no critical points or parabolic periodic points. Also,
  no point of $K$ is the accumulation point of a recurrent critical orbit,
  or of a singular orbit contained in wandering domains. 

 We note that, since $\P(f)$ is bounded, the set of recurrent critical
  points is finite, and a singular value can be recurrent
  only if it is the image of a recurrent critical point. Furthermore,
  every nonrecurrent singular value is also strongly nonrecurrent. 
  Hence every $z_0\in K$ satisfies the assumptions of Theorem
  \ref{thm:maneexplicit}, and therefore is regular in the sense of
  Definition \ref{defn:regularity}. Let $\eps>0$ be the Euclidean distance
  between $K$ and the set of critical points of $f$; by Lemma
  \ref{lem:pullbacks}, there is a disk $U$ around $z_0$ 
  such that every 
  component of $f^{-n}(U)$ that intersects $K$ is mapped properly
  by $f^n$ and does
  not pass through any critical points. Hence every such pullback
  is univalent, and any point $z\in K$ that enters the disk $U$ infinitely
  many times belongs to $J_r(f)$.

 Since every orbit in $K$ will accumulate on some $z_0\in K$, we see that
  $K\subset J_r(f)$. The claim that $K$ is hyperbolic now follows
  from the preceding lemma.
\end{proof}

\subsection*{Branched versions of $J_r(f)$}
 In view of Ma{\~n}\'e's theorem, we might want to consider
  analogs of the definition of $J_r(f)$ that do not require univalent
  pullbacks, but allow a bounded degree of branching. (Compare also
  \cite{martinmayer}, where this is the definition used for the conical
  Julia set.) 

 \begin{defn}[$\Delta$-branched radial Julia set]
  Let $\Delta\in\N$. We denote by $J_r^{\Delta}(f)$ the set of points
   $z\in J(f)$ with the following property. There is a number $\delta>0$
   such that, for infinitely many $n$, the component of
   $f^{-n}(\D^{\#}_{\delta}(f^n(z)))$ containing $z$ is simply
   connected and mapped by
   $f^n$ as a proper map of degree at most $\Delta$. 
 \end{defn}
 Note that, if $z_0$ is a regular point in the
  sense of Definition \ref{defn:regularity} and the orbit of 
  $z$ is bounded and accumulates on 
  $z_0$, then, by Lemma \ref{lem:pullbacks},  $z\in J_r^{\Delta}(f)$ for some
  $\Delta\in\N$. 

 At least in the cases of interest to us, the extra 
  generality does not gain us much in the set $J_r^{\Delta}(f)$.

 \begin{lem}[$J_r^{\Delta}(f)$ and $J_r(f)$] \label{lem:JrD}
  Suppose that $J_r^{\Delta}$ has positive measure for some $\Delta\in\N$.
   Then $J(f)=\C$ and almost every point in $\C$ has a dense orbit. 

 In particular, if $\P(f)\neq \Ch$, then 
  $J_r^{\Delta}(f)\setminus J_r(f)$ has zero measure.
 \end{lem}
 \begin{proof}
  We can carry through
   the proof of Theorem
   \ref{thm:ergodic} analogously (replacing Lemma \ref{lem:pullbacksshrink}
   by a suitable version for branched pullbacks). It follows that
   almost every point in $\C$ has a dense orbit. 

  If $\P(f)\neq\Ch$, then almost
   every orbit 
   enters some disk in $\Ch\setminus \P(f)$ infinitely many times,
   and this disk can be pulled back univalently as in the definition of
   $J_r(f)$. This proves the claim. 
 \end{proof}
 \begin{remark}
  The condition $\P(f)\neq\Ch$ is likely not necessary, but we only
   require the result in this case. 
 \end{remark}

\section{Line fields} \label{sec:linefields}

 A \emph{line field} $\nu$ on a set $A$ is a measurable choice of a tangent
  line in each point of $A$. Equivalently, a line field is a measurable
  Beltrami differential $\mu\, d\bar{z}/dz$ with $|\mu|=1$.
  (See \cite[Section 3.5]{mcmullenrenormalization}.) 
  Such a line field $\mu$ is \emph{invariant} under $f$ if
  $f^*(\mu)=\mu$ almost everywhere. We say that
  $f$ \emph{supports an invariant line field on the set $X$} if there
  is an invariant line field on some positive measure subset $A$ of $X$.

\begin{defn}[Univalent line field]
 The line field $\mu$ is called
  \emph{univalent} if for each point $z\in A$ 
  there is a neighbourhood  $U$ such that
  $\mu|_{U} = \phi^*(d\bar{z}/dz)$ for
  some univalent function $\phi\colon U\to \C$. 
  (Such a function $\phi$
  will be called a \emph{linearizing coordinate} for $\mu$.)
\end{defn}
\begin{remark}
  Univalence of a line field
   is a \emph{local} property; in particular, we do not
   require that the function $\phi$ is defined on a neighborhood of
   $A$. 
\end{remark}

 If the standard line field $d\bar{z}/dz$
   is invariant under some analytic function
  $f$, then $f'(z)$ must be a real constant, and hence $f$ is affine. 
  It follows that the preimages of straight horizontal and vertical lines
  under linearizing coordinates $\phi$ form two (transverse) analytic
  foliations on any set where $\mu$ is univalent. If $\mu$ is invariant
  under an analytic function $f$, 
  then these foliations are also invariant. 

 Let us establish some preliminaries regarding line fields that
  are univalent everywhere; in particular we will show that an entire
  function cannot have an invariant line field that is univalent at every
  point of $\C$.
  Both of the following statements are implicit in
  \cite[Proof of Theorem 2]{graczykkotusswiatek}, but we will state and prove
  them here explicitly for the reader's convenience. 

 \begin{lem}[Univalent line fields on simply connected surfaces]
    \label{lem:globalunivalent}
  Let $X$ be a simply connected Riemann surface (i.e.,
   $X$ is either the plane, the sphere or the disk).
   Suppose that $\mu$ is a line field that is
   locally univalent on all of $X$. Then
   $\mu$ is the pullback of the standard line field
    $d\bar{z}/dz$ under a nonconstant
    analytic function $g:X \to\C$ with no
    critical points. In particular, $X$ cannot be the Riemann sphere.
 \end{lem}
 \begin{proof}
  Let $U_1, U_2\subset X$ be simply connected. Assume that
   $U:=U_1\cap U_2$ is connected, and that
   there are functions $\phi_j:U_j\to\C$ satisfying
  \begin{equation}
    \mu = \phi_j^*(d\bar{z}/dz). \label{eqn:pullback}
  \end{equation}
  We claim that $\phi_1$ extends holomorphically to $U_1\cup U_2$,
   with the extension still satisfying (\ref{eqn:pullback}). 

  Indeed, consider the map 
   $\psi:= \phi_1\circ \phi_2^{-1}$, which maps
   $\phi_2(U)$ to $\phi_1(U)$ conformally. Then
   $\psi^*(d\bar{z}/dz)=d\bar{z}/dz$. This means that
   $\psi'$ is a real constant on $\phi_2(U_1\cap U_2)$
   and therefore wherever
   $\psi$ is defined it is equal to some affine map $A$
   of the form $A(z) = \lambda z + c$ 
   (with $c\in\C$ and $\lambda\in\R$). 
   It follows that the map
     \[ \phi(z) := \begin{cases}
                       \phi_1(z) & z\in U_1 \\
                       A(\phi_2(z)) & z\in U_2 \end{cases} \]
   is the desired holomorphic extension of $\phi_1$. 

  It now follows that $\phi_1$ can be extended analytically
   along any path $\gamma\subset X$, with the extension  satisfying
   (\ref{eqn:pullback}). (We cover $\gamma$ by neighborhoods
   $U_1, U_2,\dots, U_n$, with associated maps  $\phi_j$, 
   such that $U_j\cap U_{j+1}$ is connected. Then we apply our observation
   above.) By the monodromy theorem, this extension is defined on all of
   $X$, and we are done. 
 \end{proof}

\begin{cor}[Entire functions with invariant univalent line fields]
  \label{cor:entireunivalent}
  Let $f:\C\to\C$ be a nonconstant entire
  function, and suppose
   there is a
   univalent line field $\mu$ on $\C$ 
   that is invariant under
   $f$. 

  Then $f$ is an affine map. 
\end{cor}
\begin{proof}
 By the previous lemma, $\mu=g^*(d\bar{z}/dz)$ for some nonconstant
  entire function
  $g$ with no critical points. We may suppose that $f$ has a repelling
  periodic point
  $z_0$ with $g'(z_0)\neq 0$. (Otherwise, it follows directly that
  $f$ is affine.) By changing coordinates and
  passing to an iterate, we may  
  suppose that $z_0=0$, $f(0)=0$ and $g(0)=0$.  
  Let $\phi$ be the branch of $g^{-1}$ with
  $\phi(0)=0$. 
  Then the locally defined function
     \[ A := g\circ f \circ \phi \]
  preserves the standard line field, and hence is a global affine map:
     \[ A(z) = \lambda z, \quad |\lambda| = |f'(0)|>1, \quad\lambda\in\R. \]
  By the identity theorem, $g\circ f = A\circ g$ on all of $\C$, so 
  $f$ and $A$ are semiconjugate.

  If $z_1$ is another periodic point of $f$, say of period $n$, then
   $A^n(g(z_1))=g(z_1)$, so  $g(z_1)=0$. Hence the set of periodic points
   of $f$ is contained in the discrete set $g^{-1}(0)$, and therefore
   $f$ is affine.
\end{proof}

  Finally, we require
   some facts about push-forwards and pull-backs of univalent line fields
   by maps with a critical point. 
   We begin with the following simple observation.
 \begin{observation}[Critical pullbacks] \label{obs:pullbacks}
   Let $z,w\in\Ch$, and let $f$ be holomorphic near $z$, with
    $f(z)=w$.
   Let $\nu$ be a line field near $w$, and let
    $\mu$ be the pullback of $\nu$ under $f$.

   Suppose that
    $\mu$ is univalent near $z$. 
    Then either $f'(z)\neq 0$, and
    $\nu$ is univalent near $w$, or 
    $z$ is a simple critical point of $f$, and $\nu$ is univalent on a 
    punctured neighborhood of $w$, but not in $w$ itself.
 \end{observation}
\begin{remark}
 In the case where $f$ is not univalent, so that
  $\nu$ is the 
  push-forward of a univalent line field under a double cover, 
  we say that $\nu$ has 
  a \emph{one-pronged singularity} near $w$. 
\end{remark}
\begin{proof} By changing coordinates, we may assume that
  $z_0=w_0=0$, 
   $\mu = d\bar{z}/dz$, and $f(z)=z^d$ for some $d\geq 1$. 
  The line field $d\bar{z}/dz$ must then be invariant under
  $z\mapsto e^{2\pi i/d} z$, which is only possible if $d\in\{1,2\}$,
  as required.
\end{proof}

 We will also be using the following well-known fact:

 \begin{lem}  \label{lem:exceptional}
  There is no line field on the Riemann sphere that is univalent
   everywhere except (possibly) at a single one-pronged singularity.
 \end{lem}
 \begin{proof}
  By Lemma \ref{lem:globalunivalent}, 
   there is no globally univalent line field on 
   the sphere. 
   So assume $\mu$ is a line field with 
   a single one-pronged singularity.
   Without loss of generality, we may assume that this singularity is at 
   $\infty$. Then the line field is univalent on
   $\C$, and has as its linearizing coordinate a global entire function
   $\phi:\C\to\C$. However, this is impossible, since there cannot be
   a linearizing coordinate defined on a complete
   punctured neighborhood of a one-pronged singularity (i.e.,
   it is impossible to define the square root continuously on 
   a punctured neighborhood of zero). 
 \end{proof}

\section{Exceptional values}  \label{sec:exceptional}

 The big Picard theorem states that a holomorphic function cannot
  omit more than two values in the neighborhood of an essential singularity,
  and Montel's theorem states that any family of functions that
  omit the same three values is normal. We will require the
  following well-known fact from Nevanlinna theory
  (compare e.g.\ \cite{walterbloch}), which generalizes both theorems.

 \begin{thm}[Branched Values] \label{thm:nevanlinna}
   Let $a_1,\dots,a_r\in\Ch$ be distinct points, 
   and let $\nu_1,\dots,\nu_r\geq 1$ be
    integers such that
     \[ \sum_{i=1}^r\left( 1-\frac{1}{\nu_i}\right) > 2. \]
   \begin{enumerate}
    \item Let $f:\C\to \Ch$ be a
      meromorphic function such that
      each $a_i$ has only finitely many preimages of multiplicity
      less than $\nu_i$. Then $f$ is rational. \label{item:essential}
    \item Let $U\subset\C$ be a domain.
      The family of all meromorphic functions $f:U\to \Ch$
      for which 
      each preimage of
     each $a_i$ has multiplicity at least $\nu_i$ is a normal family
      in the sense of Montel.  \label{item:normal} \qedd
   \end{enumerate}
  \end{thm}

Recall that the \emph{backward orbit} of $z\in\Ch$ under $f$ is the
 set 
  \[ O^-(z) = \{w\in\C: f^n(w)=z\text{ for some } n\geq 0 \} \]
 of iterated preimages of $z$ under $f$. If $w_0\in O^-(z)$, then 
 the \emph{multiplicity} of $w_0$ as an iterated preimage of $z$ is
 its multiplicity as a zero of the function $w\mapsto f^n(w)-z$
 (where $n$ is chosen as small as possible). If the 
 multiplicity of $w_0$ is $1$, we say that the preimage $w_0$ is
 \emph{unbranched}; otherwise, $w_0$ is \emph{branched}. 

A value $z\in\Ch$ is \emph{Fatou exceptional}
 if its backward orbit  is a finite set.
 We denote
 the set of all Fatou exceptional values by $E_F(f)$.
 Let us also introduce a related concept: we say that $z$ is
 \emph{branch exceptional} if 
 $z$ has only finitely many
 unbranched iterated preimages. 
 We denote the
 set of all such values by $E_B(f)$. (This concept seems to appear implicitly
 in Schwick's simplified proof of the density of repelling periodic
 cycles in the Julia set \cite{schwick}. Also, it is 
 somewhat related to that of \emph{univalently omitted values} 
 introduced in
 \cite{graczykkotusswiatek}.)

 If $z\in E_B(f)$, 
 let us define 
  the \emph{exceptional multiplicity} 
of $z$ as the largest
  number $2\leq \nu\leq\infty$ 
  such that $O^-(z)$ contains at most finitely many points of
  multiplicity 
  less than $\nu$. 
  (In particular, the Fatou exceptional values
  are exactly the branched exceptional values of
  exceptional multiplicity $\infty$.)

The following lemma is well-known for Fatou exceptional points
 \cite[Lemma 4.9 and Theorem 4.10]{jackdynamicsthird}, but
 the corresponding statement about branch exceptional points does not
 seem to appear explicitly in the literature. 
\begin{lem}[Unbranched preimages of non-exceptional points are dense]
   \label{lem:unbranched}
 Let $f:\C\to\Ch$ be a nonconstant, nonlinear meromorphic function. For
  every branch exceptional value $v\in E_B(f)$, let $2\leq \nu(v)\leq\infty$ 
  be its exceptional multiplicity. Then
    \[ \sum_{v\in E_B(f)} \left(1-\frac{1}{\nu(v)}\right) \leq 2. \]
  In particular, $\# E_B(f)\leq 4$. 

  Furthermore, let $v\in\Ch$, and let $z\in J(f)$. If $v\notin E_F(f)$,
   then every neighborhood $U$ of $z$ contains a point $w\in O^-(v)$.
   If additionally $v\notin E_B(f)$, then
   $w$ can be chosen to be unbranched. 
\end{lem}
\begin{remark}
The first part of the theorem can also be inferred directly from
 Theorem \ref{thm:nevanlinna} for transcendental functions, while for
 rational functions it follows from elementary combinatorial considerations
 (see e.g.\ \cite[Lemma 2.3]{localcompactness}).
 Instead, we provide a simple unified proof of both statements, using
 Theorem \ref{thm:nevanlinna} in all cases. 
\end{remark}
\begin{proof}
 The theorem will be deduced from the following claim.
\begin{claim}
 Let $z\in J(f)$, and let $U$ be a neighborhood of $z$.
  Let $a_1,\dots,a_r\in\Ch$ be distinct points
  and suppose that there are
  $\nu_1,\dots,\nu_r\in \{2,3,\dots,\infty\}$ such that
  $O^-(a_i)\cap U$ contains no point of multiplicity less than
  $\nu_i$. 
  Then 
  $\sum_{i=1}^r (1 - 1/\nu_i) \leq 2$. In particular, $r\leq 4$. 
\end{claim}
\begin{subproof}
 Let us first suppose that all iterates $f^n$ can be defined 
  as meromorphic functions on $U$
  (this is always the case if $f$ is rational
  or entire). Since $z\in J(f)$, these iterates do not
  form a normal family, and the claim follows directly from
  Theorem \ref{thm:nevanlinna} (\ref{item:normal}).

 Otherwise, $f$ is a transcendental meromorphic function, and $U$ contains
  a point $z_1$ with $f^n(z_1)=\infty$ for some $n\geq 0$. We can pick a
  small neighborhood $V\subset U$ of $z_1$ that contains no
  critical points or poles of $f^n$
  apart from $z_1$. Let $W:= f^n(V)$; then $W$ is a neighborhood of $\infty$.
  By assumption, every point in $f^{-1}(a_i)\cap W$
  has   multiplicity at least $\nu_i$. 
  Hence, the claim now follows 
  from Theorem \ref{thm:nevanlinna} (\ref{item:essential}). 
\end{subproof}
To deduce the first part of the theorem, let 
 $a_1,\dots,a_r\in E_B(f)$. Since the Julia set is uncountable,
 we can choose some $z\in J(f)$ that is not one of the finitely many 
 iterated preimages of $a_i$ of multiplicity less than $\nu(a_i)$. 
 We can now pick a small neighborhood $U$ of $z$ and apply the Claim.

To prove the second part, suppose $z\in J(f)$ and $v\notin E_F(f)$ (resp.\
   $v\notin E_B(f)$). We can apply the Claim to
  five distinct iterated preimages (resp.\ unbranched iterated preimages)
  $v_1,\dots,v_5$ of $v$. It follows that at least one of these
  must have an unbranched iterated preimage in $U$, as desired. 
\end{proof}

\section{Univalent line fields} \label{sec:univalent}
 With these preliminaries, we are ready to prove the 
  main fact required for the proof of Theorem \ref{thm:mainlinefields}; compare
  \cite[Lemma 3.16]{mcmullenrenormalization}.

 \begin{thm}[Maps supporting a univalent line field]
    \label{thm:nounivalentlinefields}
  Let $f:\C\to\Ch$ be a nonconstant, nonlinear meromorphic function,
   and suppose that there is an invariant line field $\nu$ on $\Ch$
   that is univalent on an open set $U$ intersecting $J(f)$.

  Then $f$ is conjugate to one of the following:
   \begin{enumerate}
    \item $z\mapsto z^k$, $k\in\Z$, $|k|\geq 2$,
    \item $z\mapsto T_k(z)$ or $z\mapsto -T_k(z)$, where $k\geq 2$ and
      $T_k$ is the $k$-th Chebyshev polynomial, or
    \item a Latt\`es map.
   \end{enumerate}
  (In particular, $f$ is a rational function.)
 \end{thm}
 \begin{proof}
  Let $\wt{U}$ be the set of points at which the line field
   $\nu$ is univalent, and set $K := \Ch\setminus \wt{U}$. 
   Any point that has an unbranched iterated preimage in 
   $U$ must clearly belong to
   $\wt{U}$. So it follows from Lemma \ref{lem:unbranched} that
   $K\subset E_B(f)$. 

  Let us divide $K$ into two sets, by letting $K_B$ consist of all
   points of $K$ that have some iterated preimage under $f$ that belongs to 
   $\wt{U}$, and setting
   $K_F = K\setminus K_B$. Since $K$ is finite, we have $K_F\subset E_F(f)$.
   Now let $w\in K_B$ and let 
   $z\in \wt{U}$ be such that $f^n(z)=w$ for some $n\ge 1$.
   By 
   Observation    \ref{obs:pullbacks}, \, $f^n$ has a simple critical point 
    at $z$ and $\nu$ has a one-pronged singularity 
    near $w$.
    The same argument, applied to $f^{n+1}$,
    also shows that $f(K_B\setminus \{\infty\})
    \subset K_B$.     

  We set $S := \Ch\setminus K_F$ and summarize what we know so far:
   \begin{enumerate}
    \item The line field $\nu$ is univalent at every point of
      $S\setminus K_B$.
      \item $f(K_B\setminus \{\infty\})\subset K_B$.
    \item The line field $\nu$ has a one-pronged 
        singularity in every point of $K_B$.
    \item $f:S\setminus \{\infty\}\to S$ is holomorphic and preserves
      the line field $\nu$. In particular, $f$ has only simple critical
      points in $S$, and these are exactly the points in
       $f^{-1}(K_B)\setminus K_B$.   
   \end{enumerate}

 Points in $K_F$ have exceptional multiplicity $\infty$, while
  points in $K_B$ have exceptional multiplicity $2$. Hence, 
  according to Lemma~\ref{lem:unbranched}, we only have the following 
   possibilities:
    \begin{enumerate}
     \item  $S$ is the
      punctured plane, and $K_B$ is empty;
     \item $S$ is the plane, and $0\leq \# K_B \leq 2$;
     \item $S$ is the sphere, and $2\leq \# K_B \leq 4$. 
    \end{enumerate}
   (Recall that the cases $S=\Ch$ and $\# K_B = 0,1$ cannot occur by
    Lemma \ref{lem:exceptional}.)

    This amounts to saying that $S$ is an affine
    Thurston
   \emph{orbifold} with only simple branch points, and 
     $f:S\setminus\{\infty\}\to S$ is an analytic map between orbifolds.
     (Compare \cite[Appendix A]{mcmullenrenormalization} or 
     \cite[Appendix E]{jackdynamicsthird}.) 
     Because of the above classification, the
     \emph{orbifold Euler characteristic} 
     is nonnegative (where we use the terminology of 
     \cite[Appendix E]{jackdynamicsthird}), and therefore
     the universal covering $X$ is conformally $\C$ or
     $\Ch$.

  In other words, there is a ``universal cover''
    $X\in\{\C,\Ch\}$ and an analytic function
    $\pi:X\to S$ that is completely ramified of degree $2$
    over all points of
    $K_B$, and a covering elsewhere (compare
    \cite[Appendix E]{jackdynamicsthird} for the complete list of these
    covering maps). 
    Since $\nu$ lifts to
    a univalent line field $\tilde{\nu}$ on all of $X$, 
    it follows from Lemma \ref{lem:globalunivalent}
    that we must have $X=\C$. According to
    \cite[Appendix E]{jackdynamicsthird} this leaves only four cases:
    $S$ is the punctured plane and $\pi$ is the exponential map;
    $S$ is the plane and $\pi$ is a cosine map,
    $S$ is the sphere and $\pi$ is a Weierstra{\ss} $\wp$-function,
    or $S$ is the plane and $\pi$ is the identity. By Corollary
    \ref{cor:entireunivalent} and the assumption that $f$ is not affine,
    the last case does not occur.

   \begin{claim}
     $f$ lifts to an affine function under $\pi$.
     That is, there is an affine function $A:\C\to\C$ such that
      $f\circ\pi = \pi\circ A$.
   \end{claim}
  \begin{subproof}
   If $\infty\notin S$ (in particular, when $f$ is entire), then
    $f$ is a self-map of $S$, and hence can be lifted via the
    universal cover $\pi$ to a holomorphic function
    $\hat{f}:\C\to\C$ with $f\circ\pi = \pi\circ \hat{f}$. This lift 
    $\hat{f}$
    preserves the univalent line field $\tilde{\nu}$ on $\C$, and hence
    is affine by Corollary \ref{cor:entireunivalent}.

   So suppose that $\infty\in S$.
    The main problem is to show that $f$ has no asymptotic values in $S$. 
    To do so, recall that a univalent line field gives rise to 
    \emph{horizontal} and \emph{vertical} foliations
    (corresponding to straight horizontal and vertical lines under
    any linearizing coordinate of the line field $\mu$). A map that
    preserves the line field also preserves these foliations. Let us say that
    $a\in S$ is a \emph{leafwise asymptotic value} of $f$ if there is a
    piece $\gamma:[0,1)\to S\setminus\{\infty\}$ of a horizontal or
    vertical leaf with $\gamma(t)\to\infty$ and $f(\gamma(t))\to a$ 
    as $t\to 1$.    Since we assumed that $\infty\in S$, $f$ has at most four 
    leafwise asymptotic values of $f$ (one for each direction in
    which a horizontal or vertical leaf can approach $\infty$).

   We claim that any asymptotic value in $\wt{U}$ must be a leafwise 
    asymptotic value. Indeed, let $a\in\wt{U}$ be arbitrary and suppose
    that $a$ is not a leafwise asymptotic value. Let $Q\subset\wt{U}$ 
    be a neighborhood
    of $a$ that corresponds to a square under the linearizing coordinates
    of the linefield $\mu$ near $a$ and that contains no leafwise asymptotic
    values. (Recall that $\wt{U}$ also contains no critical values of $f$.) 

   The definition of a leafwise asymptotic value ensures that any branch
    of $f^{-1}$ at a point in $Q$ can be continued analytically along any
    vertical or horizontal leaf of the foliation in $Q$. But any germ
    that can be continued analytically along every horizontal and
    vertical line in a square can be continued analytically to the
    entire square. Hence every branch of $f^{-1}$ defined
    at a point of $Q$
    can be extended to the entire neighborhood 
    $Q$. This means that $Q$ contains no singular values of $f$.
    In particular, $a$ is not a singular value of $f$, as claimed. 

   So $f$ has at most finitely many asymptotic values. Suppose that
    $a\in S$ was an asymptotic value, and let $W\subset S$ 
    be a small simply connected neighborhood
    of $a$ that contains no other critical or asymptotic values of $f$. 
    Then every component of $f^{-1}(W)$ is either mapped to $W$ by a proper
    map with at most one critical point (in which case this component is 
    bounded), or to $W\setminus\{a\}$ by a universal covering. Since we 
    assumed that $a$ is an asymptotic value, there is at least one
    component $V$ of the latter kind. However, if $\gamma$ is 
    a curve in $W$ tending to $a$ along a (horizontal or vertical) leaf,
    then this curve has countably many preimages in $W$, all tending to 
    $\infty$. Since each of these preimages must itself be contained in 
    a leaf of the foliation, this is clearly incompatible with
    the structure of
    the line field near infinity.

   So $f$ has no asymptotic values. 
     Hence $f\circ\pi$ is a covering map branched exactly over $K_B$, with
     all branched points being of degree $2$. In other words, $f\circ\pi$ is
     a covering from its domain of definition to
     the orbifold $S$. Therefore the map $\pi:\C\to S$ lifts to a map
     $B:\C\to\C$ with $f\circ \pi \circ B = \pi$. This map $B$ must 
     then preserve the univalent line field $\pi^*(\mu)$ on $\C$, and hence
     is affine. $A=B^{-1}$ is the desired map. (Note that, in particular,
     $f$ must be a rational map.)
  \end{subproof}

To complete the proof of Theorem \ref{thm:nounivalentlinefields}, 
 it suffices to examine the three possible choices of $\pi$. If
 $\pi$ is an exponential map, then $f$ is a power map; if
 $\pi$ is a cosine map, then $f$ is a Chebyshev polynomial, and if
 $\pi$ is a $\wp$-function, then $f$ is a Latt\`es map.%
 \end{proof}

\section{Proofs of the main theorems} \label{sec:proofs}

\begin{proof}[Proof of Theorem \ref{thm:mainlinefields}]
 Let $z\in J_r(f)$ be a density point of the line field,
  and let $D$ be a disk of univalence
  as in Lemma \ref{lem:pullbacksshrink}. 
  By \cite[Theorem 5.16]{mcmullenrenormalization}, the line field is
  univalent on $D$. (Roughly speaking, the pushforward of the line field
  on $D_n$ to $D$ will be very close to a univalent line field, and the claim
  follows from a compactness argument.) 

 The claim now follows from Theorem \ref{thm:nounivalentlinefields}.
\end{proof}

For future reference, we also note the following restatement of Theorem
 \ref{thm:mainlinefields}:
 \begin{cor}[No line fields on $J_r(f)$] \label{cor:nolinefieldsradial}
  Let $f:\C\to\Ch$ be a nonconstant, nonlinear entire or meromorphic 
   function, and suppose that $f$ is not a Latt\`es map.

  Then $f$ supports no invariant line fields on its radial Julia set
   $J_r(f)$.
 \end{cor}
 \begin{proof}
  If $J_r(f)$ has measure zero, then $f$ supports no invariant line fields
   on $J_r(f)$ by definition. Otherwise, $f$ is measurably transitive,
   and the claim follows from Theorem \ref{thm:mainlinefields}. 
 \end{proof}

To begin the proof of Theorem \ref{thm:nonrecurrent}, 
 we note that, under fairly general hypotheses, the escaping set of
 a meromorphic function does not support invariant line fields. 
\begin{thm}[Invariant line fields on escaping sets] 
  \label{thm:nolinefieldsescaping}
 Let $f:\C\to\Ch$ be a transcendental entire or meromorphic function.
  Suppose that all poles of $f$ have degree at most $\Delta\in\N$, and that 
  the set $S(f)\setminus\{\infty\}$ 
  is bounded.

 Then the escaping set
    \[ I(f) := \{z\in\C: f^n(z)\to\infty\} \]
 supports no invariant line fields. 
\end{thm}
\begin{proof}
 Let $K>0$ be sufficiently large so that $|s|<K$ for all 
  $s\in S(f)\setminus\{\infty\}$ and set
  $D_0 := \{|z|>K\}\cup\{\infty\}$.
  If $U$ is
  a component of $f^{-1}(D_0)$, then either $U$  is
  a \emph{logarithmic tract},
  i.e.\ $f:U\to D_0\setminus\{\infty\}$ is a universal covering, or
  $U$ is bounded and $f:U\to D_0$ is a branched covering with at
  most one critical point (this critical point, if it exists, being
  a multiple pole). By assumption, the degree of the map
  $f:U\to D_0$ is bounded by $\Delta$.

 It is proved in \cite{boettcher} that there are no invariant
  line fields supported on the set of escaping points 
  $z\in I(f)$ for which $f^n(z)$ is contained in a logarithmic tract 
  for all sufficiently large $n$. We will denote this set by
  $I_{\ell}(f)$. 

 For $R>0$, let us denote by $I_R$ the set of escaping points $z\in I(f)$ that
  satisfy $|f^j(z)|\geq R$ for all $j\geq 0$. We claim that, for sufficiently
  large $R$, 
  \begin{equation} I_R \setminus I_{\ell}(f)\subset J_r^{\Delta}(f).
   \label{eqn:inclusion}
  \end{equation}

 To prove this, let $D=\{|z|>2K\}\cup\{\infty\}\subset D_0$ be
  a second disk around $\infty$, and let $C=C(\Delta,1/2)$ be the
  constant from Lemma \ref{lem:diam}. In the following, we will use
  the hyperbolic metric $\dist_{D_0^*}$ in the multiply-connected
  domain $D_0^* := D_0\setminus\{\infty\}$. If $V\subset D_0^*$ is
  a simply-connected domain, we denote by $\wt{\diam}_{D_0^*}(V)$ the
  diameter of the preimage components of $V$ in the universal cover
  of $D_0^*$. More precisely, let $\tilde{V}$ be a component of
  $\exp^{-1}(V)$; then $\wt{\diam}_{D_0^*}(V)$ is the diameter of
  $\tilde{V}$ in the hyperbolic metric of the half plane $\{\re z > \log K\}$. 

 \begin{claim}
  Let $R>0$ be chosen large enough. Suppose that $|z|\geq R$ and
   $f(z)\in D$. Let $W$ be the component of $f^{-1}(D)$ containing
   $z$. 
  \begin{enumerate}
   \item Suppose that $W$ is bounded. Then $\wt{\diam}_{D_0^*}(W)\leq C$. 
     \label{item:boundedcomponents}
   \item Suppose that $W$ is unbounded, and let $V\ni f(z)$ be a simply
    connected domain with $\wt{\diam}_{D_0^*}(V)\leq C$. Then the component
    $V'$ of $f^{-1}(V)$ containing $z$ also satisfies
    $\wt{\diam}_{D_0^*}(V')\leq C$.
      \label{item:logarithmictracts}
  \end{enumerate}
 \end{claim}
 \begin{subproof}
  If $R$ is sufficiently large, then every bounded component $W'$ of 
   $f^{-1}(D_0)$ that contains a point of modulus at least $R$ is contained
   in $D_0$. (This is  because only finitely many components of 
   $f^{-1}(D_0)$ can intersect the compact set $\C\setminus D_0$.) 

  So, if $W$ is as in (\ref{item:boundedcomponents}), then the component
   $W'$ of $f^{-1}(D_0)$ that contains $W$ satisfies $W'\subset D_0$.
   By Lemma \ref{lem:diam}, the hyperbolic diameter of $W$ in $W'$
   is bounded by $C$, and the claim follows. 

  In the case of (\ref{item:logarithmictracts}), the map
   $f:W\to D$ is a universal covering. It follows from the standard
   estimate on the hyperbolic metric in a simply connected domain
   \cite[Corollary A.8]{jackdynamicsthird} 
   that the derivative of $f$ in $z$, measured in the
   hyperbolic metric of $D_0$, tends to $\infty$ as $|z|\to\infty$ in
   $W$. (Compare \cite[Lemma 1]{alexmisha} and 
   \cite[Formula (2.4)]{boettcher}.) 
   The claim follows. 
 \end{subproof}

 Suppose in the following
  that $R$ is chosen sufficiently large according to the claim,
  and additionally that every point $z$ with $|z|>R$ has 
  $\dist_{D_0^*}(z,\partial D)>C$. 

 Now, to prove (\ref{eqn:inclusion}),
  let $z\in I_R\setminus I_{\ell}(f)$. Then there are infinitely many
  $n$ such that $f^n(z)$ is not contained in a logarithmic tract. Let
  $V_{n+1} := D$, and let $V_j$, for $j=0,\dots,n$,  be inductively defined as
  the component of $f^{-1}(V_{j+1})$
  containing $f^j(z)$. By assumption, $V_n$ is bounded, and hence 
  $\wt{\diam}_{D_0^*}(V_n)\leq C$ by
  (\ref{item:boundedcomponents}). In particular, $V_n\subset D$. 

  It now follows inductively from the claim that
   $\wt{\diam}_{D_0^*}(V_j)\leq C$ for $j=n,n-1,\dots,0$. The 
   map $f^n:V_0\to V_n$ is a conformal isomorphism, so
   $f^{n+1}:V_0\to V_{n+1}$ has degree at most $\Delta$. 

 In summary, there are infinitely many $n$ such that 
  $f^{n+1}$ takes some simply connected neighborhood of $z$ to 
  $D$ as a proper map of degree at most $\Delta$. By definition, this
  means that $z\in J^{\Delta}_r(f)$, as desired. 

 Since every escaping point $z\in I(f)$ will eventually map to a point in
  $I_R$, we see that every point of 
  $z\in I(f)\setminus I_{\ell}(f)$ either belongs to 
  $J_r^{\Delta}(f)$ or is on the backward orbit of a critical point. 

  It follows that $I(f)\setminus I_{\ell}(f)$ has zero Lebesgue measure.
  Indeed, otherwise the set $J_r^{\Delta}(f)$ would have positive
  measure, and by Lemma \ref{lem:JrD} $I(f)$ then has 
  zero measure --- a contradiction. 
  This completes the proof. 
\end{proof} 
  \begin{remark}[Remark 1]
   If we do not assume that there is a bound on the
    degree of the poles of $f$, then it need no longer be true that
    $I(f)\setminus I_{\ell}(f)$ has zero Lebesgue measure. Indeed,
    Bergweiler and Kotus \cite[Theorem 1.4]{bergweilerkotus}
    construct a transcendental
    meromorphic function $f$ for which $S(f)\setminus\{\infty\}$ is bounded 
    such that 
    $\infty$ is not an asymptotic value of $f$ but $I(f)$ has positive
    measure.
  \end{remark}
  \begin{remark}[Remark 2]
   We note that the proof of Theorem~\ref{thm:nolinefieldsescaping}
    can be significantly simplified if we assume the hypotheses of Theorem 1.1
    (which is the setting in which we will apply it).
  \end{remark}
 For future reference, we also note the following generalization of
  Theorem~\ref{thm:nolinefieldsescaping}.

\begin{thm}[Invariant line fields on generalized escaping sets] 
  \label{thm:nolinefieldsescapinggeneral}
 Let $f:\C\to\Ch$ be a transcendental entire or meromorphic function,
  and let $P$ be a finite set of pre-poles of $f$, together with $\infty$,
  such that $f(P\setminus\{\infty\})\subset P$. 
  For $p\in P$, let $k_p\geq 0$ be such that $f^{k_p}(p)=\infty$, and 
  set $k := \max_p k_p$.

 Suppose that $S(f)\setminus P$ is compact. Suppose
  furthermore that the local degree of $f$ near any point of
  $f^{-1}(P)$ is uniformly bounded by $\Delta\in\N$.

 Then the set
    \[ I^P := I^P(f) := \{z\in\C: \dist^{\#}(f^n(z),P)\to 0\} \]
 supports no invariant line fields. 
\end{thm}
 \begin{remark}
  For an example of a family where a suitable $I^P(f)$ has full measure
   but $I(f)$ has zero measure, see \cite{kotusurbanskipoles}.
 \end{remark}
\begin{proof}
 The proof is analogous to that of the previous theorem: indeed,
  we may consider it as an application of that theorem to an appropriate
  renormalization of $f$. 

 More precisely, let $D_0$ be a spherical disk around $\infty$.
  For $p\in P\setminus\{\infty\}$, let 
  $D_p$ be the component of $f^{-k_p}(D_0)$ containing $p$. If $D_0$ was
  chosen sufficiently small, then $f^{k_p}:D_p\to D_0$ is a proper
  map, unbranched except possibly over infinity, $D_p\setminus\{p\}$ 
  contains no
  singular values of $f$, and $D_p\cap D_0=\emptyset$. 

 Now let $G_p$ be the union of all components of $f^{-1}(D_p)$ that
  are contained in $D_0$, and $G:=\bigcup_p G_p$. We define
  $g:G\to D_0$ by 
    \[ g|_{G_p} := f^{k_p+1}|_{G_p}. \]
  Then every component $U$ of $G$ is simply connected, and
   either $g:U\to D_0$ is a proper map, unbranched except possibly over
   $\infty$, or $g:U\to D_0\setminus\{\infty\}$ is a universal covering. 
   Note that, in the former case, the degree of this proper map is
   uniformly bounded by $\Delta'' := \Delta\cdot \Delta'$, where
   $\Delta'$ is the maximal order of a pre-pole in $E$. 

  We can now apply the same argument as in the previous theorem
   to show that the escaping set $I(g)$ --- i.e., those orbits that
   tend to infinity under iteration --- supports no invariant
   line fields. (Note that the results of
   \cite{boettcher} do not require the function to be globally defined;
   see \cite[Corollary 4.3]{boettcher}.)

  An invariant line field for $f$ on a positive measure subset of
   $I^P$ would restrict to an invariant line field for $g$ on a positive
   measure subset of $I(g)$, so the theorem is proved.   
\end{proof}

We now prove Theorem \ref{thm:nonrecurrent}. Let us
 restate it here with the  slightly weaker hypotheses of prohibiting
 only ``bad'' wandering domains in the sense of Section \ref{sec:mane}.
 (Recall the remark
 in the introduction following Theorem \ref{thm:nonrecurrent}.)

\begin{thm}[Absence of line fields for nonrecurrent maps]
   \label{thm:nonrecurrent2}
 Let $f:\C\to\Ch$ be a nonlinear and nonconstant meromorphic function. If $f$
  is nonrecurrent and has no pre-poles of arbitrarily
  high order or bad wandering domains,
  then 
  the Julia set of $f$ supports no invariant line fields. 
\end{thm}
 \begin{proof}
 We know that there are no invariant line fields on 
  $J_r(f)$ and $I(f)$
  by Corollary~\ref{cor:nolinefieldsradial} and 
  Theorem~\ref{thm:nolinefieldsescaping}.
 We also know that, for all $\Delta\in\mathbb{N}$, the set
  $J_r^{\Delta}(f)\setminus J_r(f)$ has zero measure
 by Lemma \ref{lem:JrD}.
 Hence the theorem will be established once we prove 
  the following claim, which shows 
  that $I(f)$ and $J_r^{\Delta}(f)$ cover all of $J(f)$ except for
  a countable set.

 \begin{claim}
  If $z\in J(f)$ is not a pre-pole or eventually mapped to a parabolic periodic
   orbit, then either $z\in I(f)$ or
   $z\in J_r^{\Delta}(f)$ for some
   $\Delta\in \N$. 
 \end{claim}  
 \begin{subproof}   
   It suffices to consider the case where 
    $\dist^{\#}(f^n(z),\P(f))\to 0$ (otherwise, 
             we have 
    $z\in J_r(f)$ by definition). 

  First consider the case where the orbit of $z$ is unbounded, but
   $z\notin I(f)$. Set $R:=\max_{z\in\P(f)\setminus\{\infty\}}|z|+1$. We can
   choose
   a small disk $D\subset\{z\in\C:|z|>R\}\cup\{\infty\}$ around $\infty$
   such that no unbounded component of $f^{-1}(D)$
  contains points of modulus less than $R$.
  
  By assumption, there is a sequence $(n_k)$ such that
   $f^{n_k+1}(z)\in D$ and 
   $f^{n_k+1}(z)\to\infty$, but $|f^{n_k}(z)|< R$ for all $k$.
   If $D_k$ is the component of $f^{-(n_k+1)}(D)$ containing $z$, then 
   $D_k$ is bounded and 
   $f^{n_k+1}:D_k\to D$ is a branched covering map, unbranched except
   possibly over $\infty$. Furthermore, since $f$ does not have
   pre-poles of arbitrarily high order, the degree of this map is
   uniformly bounded. So $z\in J_R^{\Delta}(f)$, for a suitable $\Delta$,
   as claimed. 

  Now suppose that the orbit of $z$ is bounded, and let
   $K$ denote the $\omega$-limit set of $z$; i.e.\ $K$ is the
   set of limit points of the sequence $(f^n(z))$. 
   Since $z$ does not map
    to a parabolic periodic point by assumption, it follows that
    $K$ contains some point $z_0$ that is not a parabolic
    periodic point. (To see this, note that no parabolic
    point can be
    an isolated point of the $\omega$-limit set of $z$, and that the set of
    all parabolic points is countable and hence contains no perfect set.) 
    The point $z_0$ satisfies the assumptions of
    Theorem
    \ref{thm:maneexplicit}. Since the orbit of $z$ is bounded, it follows from
    Lemma \ref{lem:pullbacks} that $z\in J_r^{\Delta}$ for some $\Delta\in\N$.
 \end{subproof} 
\end{proof}

\bibliographystyle{hamsplain}
\small{%
\bibliography{/Latex/Biblio/biblio}
}

\end{document}